\documentclass{article}
\usepackage{amsmath}
\usepackage{amssymb}
\newtheorem{theorem}{Theorem}

\newtheorem{example}[theorem]{Example}

\newtheorem{lemma}[theorem]{Lemma}

\newtheorem{proposition}[theorem]{Proposition}
\newtheorem{remark}[theorem]{Remark}

\newenvironment{proof}[1][Proof]{\noindent\textbf{#1.} }{\ \rule{0.5em}{0.5em}}

\newcommand{\dint}{\int}
\newcommand{\tsum}{\sum}
\newcommand{\dbigcup}{\bigcup}
\renewcommand{\nolimits}{}
\renewcommand{\limits}{}
\def\disp{\displaystyle}

\def\Limsup{\mathop{{\rm Lim}\,{\rm sup}}}

\def\tto{\;{\lower 1pt \hbox{$\rightarrow$}}\kern -10pt
\hbox{\raise 2pt \hbox{$\rightarrow$}}\;}
\def\Hat{\widehat}

\def\Bar{\overline}
\def\ra{\rangle}
\def\la{\langle}
\def\ve{\varepsilon}

\def\R{I\!\!R}

\def\ox{\bar{x}}
\def\oy{\bar{y}}
\def\oz{\bar{z}}

\def\op{\bar{p}}

\def\co{\mbox{\rm co}\,}
\def\gph{\mbox{\rm gph}\,}
\def\epi{\mbox{\rm epi}\,}

\def\dom{\mbox{\rm dom}\,}

\def\lip{\mbox{\rm lip}\,}

\def\cl{\mbox{\rm cl}\,}

\def\cone{\mbox{\rm cone}\,}

\def\dn{\downarrow}

\def\ph{\varphi}
\def\emp{\emptyset}
\def\st{\stackrel}
\def\oR{\Bar{\R}}
\def\lm{\lambda}

\def\dd{\delta}

\newcounter{lk}

\begin{document}
\begin{center}
\textbf{QUANTITATIVE STABILITY AND OPTIMALITY CONDITIONS IN CONVEX
SEMI-INFINITE AND INFINITE PROGRAMMING}\footnote{ This research
was partially supported by grants MTM2008-06695-C03
(01-02) from MICINN (Spain).}\\[3ex]
M. J. C\'{A}NOVAS\footnote{Center of Operations Research, Miguel
Hern\'{a}ndez University of Elche, 03202 Elche (Alicante), Spain
(canovas@umh.es, parra@umh.es).}, M. A. L\'{O}PEZ\footnote{
Department of Statistics and Operations Research, University of
Alicante, 03080 Alicante, Spain (marco.antonio@ua.es).}, B. S.
MORDUKHOVICH\footnote{Department of Mathematics, Wayne State
University, Detroit, MI 48202, USA (boris@math.wayne.edu). The
research of this author was partially supported the US National
Science Foundation under grants DMS-0603848 and DMS-1007132.} and
\linebreak J. PARRA\footnotemark[2]
\end{center}

{\small\textbf{Abstract.} This paper concerns parameterized convex
infinite (or semi-infinite) inequality systems whose decision
variables run over general infinite-dimensional Banach (resp.\
finite-dimensional) spaces and that are indexed by an arbitrary
fixed set $T$. Parameter perturbations on the right-hand side of
the inequalities are measurable and bounded, and thus the natural
parameter space is $l_{\infty}(T)$. Based on advanced variational
analysis, we derive a precise formula for computing the exact
Lipschitzian bound of the feasible solution map, which involves
only the system data, and then show that this exact bound agrees
with the coderivative norm of the aforementioned mapping. On one
hand, in this way we extend to the convex setting the results of
\cite{CLMP09} developed in the linear framework under the
boundedness assumption on the system coefficients. On the other
hand, in the case when the decision space is reflexive, we succeed
to remove this boundedness assumption in the general convex case,
establishing therefore results new even for linear infinite and
semi-infinite systems. The last part of the paper provides
verifiable necessary optimality conditions for infinite and
semi-infinite programs with convex inequality constraints and
general nonsmooth and nonconvex objectives. In this way we extend
the corresponding results of \cite{CLMP10} obtained for programs
with linear infinite inequality constraints.}\vspace*{0.05in}

{\small \textbf{Key words.} semi-infinite and infinite
programming, parametric optimization, variational analysis, convex
infinite inequality systems, quantitative stability, Lipschitzian
bounds, generalized differentiation, coderivatives
\vspace*{0.05in} }\vspace*{0.05in}

\bigskip

\section{Introduction}

Many optimization problems are formulated in the form:
\begin{equation*}
\begin{array}{rrl}
\mathrm{(P)} & \mathrm{inf} & \varphi (x) \\
& {\text{s.t.}} & f_{t}(x)\le 0,\;t\in T,
\end{array}
\end{equation*}
where $T$ is an arbitrary \emph{index set}, where $x\in X$ is a
{\em decision variable} selected from a general Banach space $X$
with its topological dual denoted by $X^*$, and where
$f_{t}:X\rightarrow\overline{\mathbb{R}}:=\mathbb{R}\cup\{\infty
\},$ $t\in T$, are proper lower semicontinuous (lsc) convex
functions; these are our {\em standing assumptions}. In this paper
we analyze {\em quantitative stability} of the feasible set of
$\mathrm{(P)}$ under small perturbations on the right-hand side of
the constraints. In more detail, the paper is focused on
characterizing {\em Lipschitzian behavior} of the feasible
solution map, with computing the {\em exact bound} of Lipschitzian
moduli by using appropriate tools of advanced variational analysis
and generalized differentiation particularly based on {\em
coderivatives}; see below.

In what follows we consider the \emph{parametric convex inequality
system}
\begin{equation}
\sigma(p):=\big\{f_{t}\left( x\right)\le p_{t},\ t\in T\big\},
\label{eq_conv_sys}
\end{equation}
where the functional parameter $p$ is a measurable and essentially
bounded function $p:T\rightarrow\mathbb{R}$, i.e., $p$ belongs to
the Banach space $l_{\infty }(T)$; we use the notation $p_{t}$ for
$p(t)$, $t\in T$. The zero function $\overline{p}=0$ is regarded
as the \emph{nominal parameter}. This assumption does not entail
any loss of generality.

Recall that the \emph{parameter space} $l_{\infty }(T)$ is a
Banach space with the norm
\begin{equation*}
\left\Vert p\right\Vert:=\sup_{t\in T}\left\vert p_{t}\right\vert.
\end{equation*}
If no confusion arises, we also use the same notation $\Vert\cdot
\Vert$ for the given norm in $X$ and for the corresponding dual
norm in $X^{\ast }$ defined by
\begin{equation*}
\left\Vert x^{\ast }\right\Vert :=\sup\limits_{\left\Vert
x\right\Vert \leq 1}\left\langle x^{\ast },x\right\rangle\;\text{
for any }\;x^{\ast }\in X^{\ast},
\end{equation*}
where $\left\langle x^{\ast },x\right\rangle$ stands for the
standard canonical pairing. Our main attention is focused on the
\emph{feasible solution map} $\mathcal{F}:l_{\infty
}(T)\rightrightarrows X$ defined by
\begin{equation}
\mathcal{F}(p):=\big\{x\in X\big|\;x\text{ is a solution to
}\sigma (p)\big\}. \label{eq_fsm}
\end{equation}

The convex system $\sigma (p)$ with $p\in l_{\infty }(T)$ can be
{\em linearized} by using the \emph{Fenchel-Legendre conjugate}
$f_{t}^ {\ast }:X^{\ast }\rightarrow\overline{\mathbb{R}}$ for
each function $f_{t}$ given by
\begin{equation*}
f_{t}^{\ast}\left( u^{\ast }\right):=\sup\big\{\left\langle
u^{\ast },x\right\rangle -f_{t}\left( x\right)\big|\;x\in
X\big\}=\sup\big\{\left\langle u^{\ast },x\right\rangle
-f_{t}\left(x\right)\big|\;x\in \mathrm{dom}f_{t}\big\},
\end{equation*}
where $\mathrm{dom}f_{t}:=\left\{ x\in X\mid f_{t}\left( x\right)
<\infty \right\} $ is the effective domain of $f_{t}$.
Specifically, under the current assumptions on each $f_{t}$ its
conjugate $f_{t}^{\ast }$ is also a proper lsc convex function
such that
\begin{equation*}
f_{t}^{\ast \ast }=f_{t}\;\mbox{ on }\;X\;\mbox{ with
}\;f_{t}^{\ast \ast }:=\left( f_{t}^{\ast}\right)^{\ast }.
\end{equation*}
In this way, for each $t\in T$, the inequality $f_{t}\left(
x\right)\le p_{t}$ turns out to be equivalent to the linear system
\begin{equation*}
\left\{\left\langle u^{\ast},x\right\rangle -f_{t}^{\ast }\left(
u^{\ast }\right) \leq p_{t},\text{ }u^{\ast }\in
\mathrm{dom}f_{t}^{\ast}\right\}
\end{equation*}
in the sense that they have the same solution sets. Then we
consider the following parametric family of linear systems:
\begin{equation}
\widetilde{\sigma}(\rho ):=\big\{\left\langle u^{\ast
},x\right\rangle \leq f_{t}^{\ast }\left( u^{\ast }\right) +\rho
\left( t,u^{\ast }\right) ,\ \left( t,u^{\ast }\right) \in
\widetilde{T}\big\},\label{q_lin_sys}
\end{equation}
where $\widetilde{T}:=\left\{\left( t,u^{\ast }\right)\in T\times
X^{\ast }\;|\;u^{\ast }\in \mathrm{dom}f_{t}^{\ast }\right\}$, and
the associated feasible set mapping
$\widetilde{\mathcal{F}}:l_{\infty}(\widetilde{T})\rightrightarrows
X$ given by
\begin{equation}
\widetilde{\mathcal{F}}\left(\rho \right):=\big\{x\in X\big|\
\;x\text{ is a solution to }\;\widetilde{\sigma}(\rho)\big\}.
\label{eq_feasible_tilde}
\end{equation}
Thus our initial family $\left\{\mathcal{F}\left( p\right),~p\in
l_{\infty}(T)\right\}$ can be straightforwardly embedded into the
family $\{\widetilde{\mathcal{F}}\left(\rho\right),~\rho \in
l_{\infty }(\widetilde{ T})\}$ through the relation
\begin{equation}
\mathcal{F}\left( p\right)=\widetilde{\mathcal{F}}\left( \rho
_{p}\right)\;\text{ for }\;p\in l_{\infty
}(T),\label{eq_equality_feasible}
\end{equation}
where $\rho_{p}\in l_{\infty }(\widetilde{T})$ is defined by
\begin{equation*}
\rho_{p}\left(t,u^{\ast }\right) :=p_{t}\;\text{ for }\;\left(
t,u^{\ast }\right) \in \widetilde{T}.
\end{equation*}
We consider the supremum norm in $l_{\infty}(\widetilde{T})$,
which for the sake of simplicity is also denoted by $\Vert\cdot
\Vert$. Note that
\begin{equation*}
\left\Vert p\right\Vert =\sup_{t\in T}\left\vert p_{t}\right\vert
=\sup_{\left( t,u\right) \in \widetilde{T}}\left\vert \rho
_{p}\left( t,u\right) \right\vert =\left\Vert \rho
_{p}\right\Vert.
\end{equation*}

Since the $f_{t}$'s are fixed functions, the structure of
$\widetilde{\sigma }(\rho )$, $\rho \in l_{\infty
}(\widetilde{T})$, fits into the context analyzed in
\cite{CLMP09}, and some results of the present paper take
advantage of this fact. The implementation of this idea requires
establishing precise relationships between Lipschitzian behavior
of $\mathcal{F}$ at the nominal parameter $\overline{p}=0$ and
that of $\widetilde{\mathcal{F}}$ at $\rho_{\overline{p}}=0$,
which is done in what follows. This approach allows us to derive
{\em characterizations of quantitative/Lipschitzian stability} of
parameterized sets of feasible solutions described by infinite
systems of {\em convex} inequalities, with computing the exact
bound of Lipschitzian moduli, from those obtained in \cite{CLMP09}
for their linear counterparts in general Banach spaces.

Furthermore, in the case of {\em reflexive} spaces of decision
variables we manage to {\em remove} the {\em boundedness}
requirement on coefficients of linear systems imposed in
\cite{CLMP09} and thus establish in this way complete
characterizations of quantitative stability of general convex
systems of infinite inequalities under the most natural
assumptions on the initial data.

Our approach to the study of quantitative stability of infinite
convex systems is mainly based on {\em coderivative analysis} of
set-valued mappings of type \eqref{eq_fsm}. As a by-product of
this approach, we derive verifiable {\em necessary optimality
conditions} for semi-infinite programs with convex inequality
constraints and general (nonsmooth and nonconvex) objective
functions.\vspace*{0.05in}

The rest of the paper is organized as follows. Section~2 presents
some basic definitions and key results from variational analysis
and generalized differentiation needed in what follows.

In Section~3 we derive auxiliary results for infinite systems of
convex inequalities used in the proofs of the main results of the
paper.

Section~4 is devoted to the quantitative stability analysis of
parameterized infinite systems of convex inequalities by means of
coderivatives in arbitrary Banach spaces of decision variables.
Based on this variational technique, we establish verifiable
characterizations of the Lipschitz-like property of the perturbed
feasible solution map \eqref{eq_fsm} with precise computing the
exact Lipschitzian bound in terms of the initial data of
\eqref{eq_conv_sys}. This is done by reducing \eqref{eq_conv_sys}
to the linearized system \eqref{q_lin_sys} in the way discussed
above.

In Section~5 we show how to remove, in the case of reflexive
decision spaces, the boundedness assumption on coefficients of
linear infinite systems and hence for the general convex infinite
systems \eqref{eq_conv_sys} via the linearization procedure
\eqref{q_lin_sys} in the above quantitative stability analysis and
characterizations.

Finally, Section~6 is devoted to deriving subdifferential
optimality conditions for semi-infinite and infinite programs of
type $\mathrm{(P)}$ with convex infinite constraints and
nondifferentiable (generally nonconvex) objectives.
\vspace*{0.05in}

Our notation  is basically standard in the areas of variational
analysis and semi-infinite/infinite programming; see, e.g.,
\cite{GoLo98,mor06a}. Unless otherwise stated, all the spaces
under consideration are Banach. The symbol $w^*$ signifies the
weak$^*$ topology of a dual space, and thus the weak$^*$
topological limit corresponds to the weak$^*$ convergence of nets.
Some particular notation will be recalled, if necessary, in the
places where they are introduced.

\section{Preliminaries from Variational Analysis}

Given a set-valued mapping $F\colon Z\rightrightarrows Y$ between
Banach spaces $Z$ and $Y$, we say the $F$ is \emph{Lipschitz-like
around} $(\bar{z},\bar{y})\in\gph F$ with \emph{modulus} $\ell
\geq 0$ if there are neighborhoods $U$ of $\bar{z}$ and $V$ of
$\bar{y}$ such that
\begin{equation}
F(z)\cap V\subset F(u)+\ell \Vert z-u\Vert \mathbb{B}\;\text{ for
any }\;z,u\in U,\label{eq_Lips}
\end{equation}
where $\mathbb{B}$ stands for the closed unit ball in the space in
question. The infimum of moduli $\{\ell\}$ over all the
combinations of $\{\ell ,U,V\}$ satisfying (\ref{eq_Lips}) is
called the \emph{exact Lipschitzian bound} of $ F$ around
$(\bar{z},\bar{y})$ and is labeled as $\lip F(\bar{z},\bar{y})$.

If $V=Y$ in (\ref{eq_Lips}), this relationship signifies the
classical (Hausdorff) \emph{local Lipschitzian} property of $F$
around $\bar{z}$ with the \emph{exact Lipschitzian bound} denoted
by $\lip F(\bar{z})$ in this case.

It is worth mentioning that the Lipschitz-like property (also
known as the Aubin or pseudo-Lipschitz property) of an arbitrary
mapping $F\colon Z\rightrightarrows Y$ between Banach spaces is
equivalent to other two fundamental properties in nonlinear
analysis while defined for the inverse mapping $F^{-1}\colon
Y\rightrightarrows X$; namely, to the \emph{metric regularity} of
$F^{-1}$ and to the \emph{linear openness} of $F^{-1}$ around
$(\bar{y},\bar{z})$, with the corresponding relationships between
their exact bounds (see, e.g. \cite{iof,mor06a,rw}). From these
relationships we can easily observe the following representation
for the exact Lipschitzian bound:
\begin{equation}
\lip
F(\bar{z},\bar{y})=\limsup_{(z,y)\rightarrow(\bar{z},\bar{y})}
\frac{\mbox{dist}\big(y;F(z)\big)}{\mbox{dist}\big(z;F^{-1}(y)\big)},
\label{eq_quotient}
\end{equation}
where $\inf\emptyset:=\infty$ (and hence $\mbox{dist}(x;\emptyset
)=\infty $) as usual, and where $0/0:=0$. We have accordingly that
$\lip F(\bar{z},\bar{y})=\infty$ if $F$ is not Lipschitz-like
around $(\bar{z},\bar{y})$.

A remarkable fact consists of the possibility to characterize
pointwisely the (derivative-free) Lipschitz-like property of $F$
around $(\bar{z},\bar{y})$---and hence its local Lipschitzian,
metric regularity, and linear openness counterparts---in terms of
a dual-space construction of generalized differentiation called
the {\em coderivative} of $F$ at $(\bar{z},\bar{y})\in\gph F$. The
latter is a positively homogeneous multifunction $D^{\ast
}F(\bar{z},\bar{y})\colon Y^{\ast }\rightrightarrows Z^{\ast }$
defined by
\begin{equation}\label{cod}
D^{\ast}F(\bar{z},\bar{y})(y^{\ast}):=\big\{z^{\ast}\in
Z^{\ast}\big|\;(z^{\ast},-y^{\ast})\in
N\big((\bar{z},\bar{y});\gph F\big)\big\},\quad y^{\ast}\in
Y^{\ast},
\end{equation}
where $N(\cdot ;\Omega)$ stands for the collection of generalized
normals to a set at a given point known as the \emph{basic}, or
\emph{limiting}, or \emph{Mordukhovich normal cone}; see, e.g.
\cite{mor76,mor06a,rw,s} and references therein. When both $Z$ and
$Y$ are finite-dimensional, it is proved in \cite{mor93} (cf.\
also \cite[Theorem~9.40]{rw}) that a closed-graph mapping $F\colon
Z\rightrightarrows Y$ id Lipschitz-like around
$(\bar{z},\bar{y})\in\gph F$ if and only if
\begin{equation}\label{cod-cr}
D^{\ast }F(\bar{z},\bar{y})(0)=\{0\},
\end{equation}
and the exact Lipschitzian bound of moduli $\{\ell \}$ in (\ref{eq_Lips}) is
computed by
\begin{equation}
\lip F(\bar{z},\bar{y})=\Vert D^{\ast}F(\bar{z},\bar{y})\Vert
:=\sup\big\{\Vert z^{\ast }\Vert \;\big|\;z^{\ast}\in
D^{\ast}F(\bar{z},\bar{y} )(y^{\ast}),\;\Vert y^{\ast }\Vert\le
1\big\}.  \label{eq_norm coderiv}
\end{equation}
There is an extension \cite[Theorem~4.10]{mor06a} of the
coderivative criterion \eqref{cod-cr}, via the so-called mixed
coderivative of $F$ at $(\oz,\oy)$, to the case when both spaces
$Z$ and $Y$ are Asplund (i.e., their separable subspaces have
separable duals) under some additional ``partial normal
compactness" assumption that is automatic in finite dimensions.
Also the aforementioned theorem contains an extension of the exact
bound formula \eqref{eq_norm coderiv} provided that $Y$ is Asplund
while $Z$ is finite-dimensional. Unfortunately, none of these
results is applied in our setting \eqref{eq_fsm}.

Indeed, the underlying set-valued mapping \eqref{eq_fsm}
considered in this paper is $\mathcal{F}\colon l_{\infty
}(T)\rightrightarrows X$ defined by the infinite system of convex
inequalities (\ref{eq_conv_sys}). The graph $\gph \mathcal{F}$ of
this mapping is obviously convex, and we can easily verify that it
is also closed with respect to the product topology. If the index
set $T$ is infinite, $l_{\infty}(T)$ is an infinite-dimensional
Banach space, which is {\em never Asplund}. There exists an
isometric isomorphism between the topological dual $l_{\infty
}(T)^{\ast }$ and the space $ba(T)$ of additive and bounded
measures $\mu\colon T\tto\mathbb{R}$ such that
\begin{equation*}
\left\langle \mu ,y\right\rangle =\dint\nolimits_{T}y_{t}\ \mu (dt).
\end{equation*}
The \emph{dual norm }$\left\Vert\mu\right\Vert$ is the total
variation of $\mu $ on $T$, i.e.,
\begin{equation*}
\left\Vert\mu \right\Vert =\sup_{A\subset T}\mu (A)-\inf_{B\subset
T}\mu (B).
\end{equation*}
All these topological facts are classical and can be found, e.g.,
in \cite{DS}.

\section{Auxiliary Results for Infinite Convex Systems}

Given a subset $S$ of a normed space, the notation $\co S$ and
$\cone S$ stand for the convex hull and the conic convex hull of
$S$, respectively. The symbol $\mathbb{R}_{+}$ signifies the
interval $\left[0,\infty\right)$, and by
$\mathbb{R}_{+}^{\left(T\right)}$ we denote the collection of all
the functions $\lambda=\left(\lambda _{t}\right)_{t\in T}\in
\mathbb{R}_{+}^{T}$ such that $\lambda_{t}>0$ for only {\em
finitely many} $t\in T$. As usual, $\mathrm{cl}^{\ast}S$ stands
for the weak$^*$ ($w^{\ast}$ in brief) topological closure of $S$.

The \textit{indicator function} $\delta_{S}=\dd(\cdot;S)$ of the
set $S$ is defined by $\delta_{S}(x):=0$ if $x\in S$ and
$\delta_{S}(x):=\infty $ if $x\notin S$. It is easy to see that
$S$ is a nonempty closed convex set if and only if $\delta _{S}$
is a proper lsc convex function. For a function $h:X\rightarrow
\overline{\mathbb{R}}$ the \emph{epigraph} of $h$ is given by
\begin{equation*}
\epi h:=\big\{(x,\gamma)\in X\times \mathbb{R}\big|\;x\in\dom h,\
h(x)\le\gamma\big \}.
\end{equation*}

The following \emph{extended Farkas' Lemma} is a key tool in our
analysis.

\begin{lemma}\label{Lem_Farkas copy(1)}\emph{(cf.\ \cite[Theorem 4.1]{DGL07})} For $p\in
\dom(\mathcal{F)}$ and $\left(v,\alpha\right)\in X^{\ast}\times
\mathbb{R}$, the following statements are equivalent:

\emph{(i) }$v(x)\leq \alpha \,$is a consequence of $\sigma (p);$
i.e., $v(x)\le\alpha $ for all $x\in\mathcal{F}\left(p\right)$.

\emph{(ii) }$\left( v,\alpha \right)\in\cl\nolimits^{\ast}\left(
\cone\dbigcup\nolimits_{t\in
T}\mathrm{epi}(f_{t}-p_{t})^{\ast}\right)$.
\end{lemma}

\begin{proof}
Theorem 4.1 in \cite{DGL07} yields the equivalence between (i) and
the inclusion
\begin{equation*}
\left(v,\alpha \right)\in\cl\nolimits^{\ast }\left(\cone
\dbigcup\nolimits_{t\in T}\epi(f_{t}-p_{t})^{\ast}+\mathbb{R}
_{+}(0,1)\right).
\end{equation*}
Thus it suffices to observe that
$(0,1)\in\cl\nolimits^{\ast}\left(\cone\dbigcup\nolimits_{t\in
T}\mathrm{epi}(f_{t}-p_{t})^{\ast }\right)$. To do this, pick any
$\left( w,\beta
\right)\in\mathrm{epi}(f_{t_{0}}-p_{t_{0}})^{\ast}$ for some
$t_{0}\in T$ and note that
\begin{equation*}
(0,1)=\lim_{r\rightarrow \infty }\frac{1}{r}\left(w,\beta
+r\right),
\end{equation*}
where the limit is taken with respect to the strong topology.
\end{proof}

\begin{remark}
\label{Rem1}\emph{As an application of the previous lemma,
together with the Br\o ndsted-Rockafellar theorem (which yields,
for each }$t\in T,$ \emph{that}
\begin{equation*}
\mathrm{rge}(\partial f_{t})\subset
\mathrm{dom}(f_{t}^{\ast})\subset\cl^{\ast }\left(
\mathrm{rge}(\partial f_{t})\right);
\end{equation*}
\emph{see, e.g. \cite[Theorem~3.1.2]{z}}$),$\emph{\ we get the
representation}
\begin{equation*}
\mathcal{F}\left(p\right)=\big\{ x\in X\big|\;\left\langle u^{\ast
},x\right\rangle -f_{t}^{\ast }\left( u^{\ast }\right)\le\ p_{t},\text{ }%
t\in T,\text{ }u^{\ast }\in \mathrm{rge}(\partial f_{t})\big\}
\end{equation*}
\emph{providing an alternative way of linearizing our convex system (\ref%
{eq_conv_sys}).}
\end{remark}

Let us now define, for $p\in l_{\infty}(T),$ the sets {\normalsize
\begin{equation}\label{H}
H\left( p\right):=\co\left( \dbigcup\nolimits_{t\in T}\epi
(f_{t}-p_{t})^{\ast}\right)\subset X^{\ast}\times \mathbb{R},
\end{equation}
\begin{equation}\label{C}
C\left( p\right):=\co \left(\bigcup\nolimits_{t\in T}\gph
(f_{t}-p_{t})^{\ast}\right) \subset X^{\ast }\times\mathbb{R}.
\end{equation}
Note that in the case of linear constraints of the type $``\ge
\textquotedblright$ the set $H(p)$ in \eqref{H} coincides with
what was called \emph{hypographical set} in
\cite{CLPT05}.\vspace*{0.05in}

We say that the system {\normalsize $\sigma\left(0\right)$}
satisfies the {\normalsize \ \emph{strong Slater condition} (SSC)}
if there exists a point {\normalsize \ }$\widehat{x}\in X$ such
that
\begin{equation*}
\sup_{t\in T}f_{t}(\widehat{x})<0.
\end{equation*}
In this case $\widehat{x}$ is called a {\em strong Slater point}
for {\normalsize $ \sigma \left( 0\right).$} Note that
$\widehat{x}$ is a strong Slater point for $\sigma \left( 0\right)
$ if and only if $\widehat{x}$ is a strong Slater point for the
linear system $\widetilde{\sigma }\left( 0\right)$, i.e.,
$\sup_{\left( t,u^{\ast }\right) \in \widetilde{T}}\{\left\langle
u^{\ast },\widehat{x}\right\rangle -f_{t}^{\ast }\left( u^{\ast
}\right) \}<0$.

\begin{lemma}
\label{Lemma1} Assume that $0\in\dom\mathcal{F}$. The following
statements are equivalent:

{\bf (i)} $\sigma\left(0\right)$ satisfies the SSC.

{\bf (ii)} $0\in{\rm{int}}(\dom\mathcal{F)}$.

{\bf (iii)} $\mathcal{F}$ is Lipschitz-like around $(0,x)$ for all
$x\in\mathcal{F}\left(0\right)$

{\bf (iv)} $(0,0)\notin\cl^{\ast}H\left(0\right)$.

{\bf (v)} $(0,0)\notin\cl^{\ast}C\left(0\right)$.
\end{lemma}

\begin{proof}
The equivalence between (i), (ii),{\normalsize \ }and{\normalsize
\ }(iv) are established in Theorem~5.1 of \cite{DGL08}{\normalsize
. } The equivalence between (ii) and (iii) follows from the
classical Robinson-Ursescu theorem. Implication (iv)$~\Rightarrow
~$(v) is obvious by the inclusion $C\left( 0\right) \subset
H\left(0\right)$ due to \eqref{H} and \eqref{C}.

Let us now check that the inclusion {\normalsize $(0,0)\in\cl
\nolimits^{\ast}H\left( 0\right)$} implies the one in {\normalsize
$(0,0)\in\cl\nolimits^{\ast}C\left( 0\right)$}, which thus yields
(v)$~\Rightarrow ~$(iv)$.$ To proceed, assume that {\normalsize
$(0,0)\in\cl\nolimits^{\ast}H\left(0\right)$} and write
\begin{equation}
{\normalsize (0,0)=w^{\ast }}\text{{\normalsize -}}{\normalsize
\lim_{\nu \in \mathcal{N}}\left\{ \tsum\limits_{t\in T}\alpha
_{t\nu }\left( v_{t\nu }^{\ast },f_{t}^{\ast}\left( v_{t\nu
}^{\ast }\right) +\beta _{t\nu }\right) \right\} } \label{eq_lim}
\end{equation}
for some net indexed by a certain directed set $\mathcal{N}$ and
satisfying the conditions
\begin{eqnarray*}
\tsum\limits_{t\in T}\alpha _{t\nu}&=&1\;\text{ for all }\;\nu
\in\mathcal{N}
,\text{ } \\
\left(v_{t\nu }^{\ast},f_{t}^{\ast}\left( v_{t\nu
}^{\ast}\right)+\beta_{t\nu }\right) &\in &\mathrm{epi}~f_{t
}^{\ast }\;\text{ for all }\;t \text{ and all }\nu\in\mathcal{N}
\end{eqnarray*}
with $\alpha _{\nu }=\left( \alpha _{t\nu }\right) _{t\in T}$ and $\beta
_{\nu }=\left( \beta _{t\nu }\right) _{t\in T}$ belonging to $\mathbb{R}%
_{+}^{\left( T\right) }$. Take then any $\overline{x}\in
\mathcal{F}\left( 0\right)$ and observe from (\ref{eq_lim}) the
relationships
\begin{eqnarray*}
0 &=&{\normalsize \lim_{\nu \in \mathcal{N}}\left\{
\tsum\limits_{t\in T}\alpha _{t\nu }\left( \left\langle v_{t\nu
}^{\ast },\overline{x} \right\rangle -f_{t}^{\ast }\left( v_{t\nu
}^{\ast }\right) -\beta
_{t\nu }\right) \right\} } \\
&\leq &{\normalsize \lim_{\nu
\in\mathcal{N}}\left\{\tsum\limits_{t\in T}\alpha _{t\nu }\left(
f_{t}\left( \overline{x}\right) -\beta _{t\nu }\right) \right\}
}\leq 0
\end{eqnarray*}
held due to the feasibility of $\overline{x}$ and the fact that
$\beta_{t\nu }\geq 0\,$\ for all $t$ and all $\nu$. Hence we
arrive at the equality
\begin{equation*}
{\normalsize \lim_{\nu \in \mathcal{N}}\left\{ \tsum\limits_{t\in
T}\alpha _{t\nu }\beta _{t\nu }\right\}}=0
\end{equation*}
yielding in turn that
\begin{equation*}
{\normalsize (0,0)=w^{\ast}}\text{{\normalsize-}}{\normalsize
\lim_{\nu \in \mathcal{N}}\left\{ \tsum\limits_{t\in T}\alpha
_{t\nu }\big(v_{t\nu }^{\ast},f_{t}^{\ast}\left( v_{t\nu
}^{\ast}\right)\big)\right\}\in}\cl\nolimits^{\ast}\left(
\co\left(\bigcup\nolimits_{t\in T}\gph f_{t}{}^{\ast}\right)
\right),
\end{equation*}
which thus completes the proof of the lemma.
\end{proof}\vspace*{0.05in}

The following two technical statements are of their own interest
while playing an essential role in proving the main results
presented in the subsequent sections. We keep the convention
$0/0:=0$.

\begin{proposition}\label{Prop_distance} Suppose that $X$ is a Banach
space and that $g:X\rightarrow\overline{\mathbb{R}}$ is a proper
convex function such that there exists $\widehat{x}\in X$ with
$g(\widehat{x})<0$. If
\begin{equation}\label{S}
S:=\big\{y\in X\big|\;g(y)\le 0\big\},
\end{equation}
then for all $x\in X$ we have the equality
\begin{equation}
\mathrm{dist}\left(x;S\right)=\sup_{{(x^{\ast},\alpha
)\in{{\rm\small epi}}\,g^{\ast}}}{\frac{\left[\left\langle
x^{\ast},x\right\rangle-\alpha \right]_{+}}{\left\Vert
x^{\ast}\right\Vert}.}\label{distepig*}
\end{equation}
\end{proposition}

\begin{proof} Observe that the nonemptiness of the set $S$ defined
in \eqref{S} ensures that $\alpha \ge 0$ whenever
$(x^{\ast},\alpha )\in\epi g^{\ast }$, and so the possibility of
$x^{\ast }=0$ is not an obstacle in (\ref{distepig*}) under our
convention that $0/0=0$. Note also that
$\mathrm{dist}\left(x;S\right) $ is nothing else but the optimal
value in the convex optimization problem
\begin{equation*}
\mathrm{\inf }\text{ }\left\Vert y-x\right\Vert\;\text{ s.t.
}\;g(y)\leq 0.
\end{equation*}
Since for this problem the classical Slater condition is
satisfied, the \emph{strong Lagrange duality} holds (see, e.g.,
\cite[Theorem 2.9.3]{z}); namely,
\begin{eqnarray*}
\mathrm{dist}\left(x;S\right)&=&\max_{\lambda \ge 0}\inf_{y\in
X}\big\{\left\Vert y-x\right\Vert +\lambda g(y)\big\} \\
&=&\max\big\{\sup_{\lambda >0}\inf_{y\in X}\left\{\left\Vert
y-x\right\Vert +\lambda g(y)\big\} ,\ \inf_{y\in X}\left\Vert
y-x\right\Vert \right\} \\
&=&\max \left\{ \sup_{\lambda >0}\inf_{y\in X}\big\{\left\Vert
y-x\right\Vert +\lambda g(y)\big\},\ 0\right\}.
\end{eqnarray*}
Applying now  the Fenchel duality theorem to the inner infimum
problem for every fixed $\lambda>0$, which is possible due to the
obvious fulfillment of the Rockafellar regularity condition, we
get
\begin{equation*}
\inf_{y\in X}\big\{\left\Vert y-x\right\Vert +\lambda g(y)\big\}
=\max_{y^{\ast }\in X^{\ast}}\big\{-\left\Vert \cdot-x\right\Vert
^{\ast }(-y^{\ast})-(\lambda g)^{\ast}(y^{\ast})\big\}.
\end{equation*}
By employing next the well-known formula
\begin{equation*}
\left\Vert \cdot -x\right\Vert ^{\ast }(-y^{\ast})=\left\{
\begin{array}{ll}
\left\langle -y^{\ast},x\right\rangle & \text{if }\left\Vert
y^{\ast}\right\Vert \le 1, \\
\infty & \text{otherwise,}
\end{array}\right.
\end{equation*}
we arrive at the relationships
\begin{eqnarray*}
\inf_{y\in X}\big\{\left\Vert y-x\right\Vert +\lambda g(y)\big\}
&=&\max_{\left\Vert y^{\ast}\right\Vert\le 1}\big\{\left\langle
y^{\ast},x\right\rangle-(\lambda g)^{\ast}(y^{\ast})\big\} \\
&=&\max_{\substack{{y^{\ast}\in X}^{\ast},\ {\rho\in\mathbb{R}},
\\ \left\Vert y^{\ast}\right\Vert\le 1,\ {(\lambda g)^{\ast}(y^{\ast})\le
\rho}}}{\big\{\left\langle y^{\ast},x\right\rangle-\rho\big\}} \\
&=&\max_{\substack{{y^{\ast}\in X}^{\ast},\ {\rho\in\mathbb{R}},\\
\left\Vert y^{\ast}\right\Vert \leq 1,\ {\lambda g^{\ast}(y^{\ast
}/\lambda)\le\rho}}}{\big\{\left\langle y^{\ast},x\right\rangle
-\rho\big\}} \\
&=&\max_{\substack{{y^{\ast}\in X}^{\ast},\ {\rho\in
\mathbb{R}},\\ \left\Vert y^{\ast}\right\Vert\le 1,\ {(1/\lambda
)(y^{\ast },\rho)\in{{\rm\small epi}}\,g^{\ast}}}}{\big\{
\left\langle y^{\ast},x\right\rangle-\rho\big\}}.
\end{eqnarray*}
Thus defining $x^{\ast}:=(1/\lambda )y^{\ast}$ and
$\alpha:=(1/\lambda )\rho $ gives us
\begin{equation*}
\inf_{y\in X}\big\{\left\Vert y-x\right\Vert+\lambda g(y)\big\}
=\max_{\substack{ x{^{\ast}\in X}^{\ast},\ \alpha {\in\mathbb{R}},
\\ \left\Vert x^{\ast}\right\Vert \leq 1/\lambda ,\ {(x^{\ast},\alpha)
\in{{\rm\small epi}}\,g^{\ast}}}}{\lambda\big\{\left\langle
x^{\ast},x\right\rangle-\alpha\big\}}\quad\mbox{and}
\end{equation*}
\begin{eqnarray}
\mathrm{dist}\left(x;S\right) &=&\max \left\{\sup_{\substack{
\lambda
>0,\ x{^{\ast}\in X}^{\ast},\ \alpha {\in\mathbb{R}},
\\ \left\Vert x^{\ast}\right\Vert \leq 1/\lambda,\ {(x^{\ast},
\alpha)\in{{\rm\small epi}}\,g^{\ast}}}}{\lambda\big\{\left\langle
x^{\ast},x\right\rangle-\alpha\big\}
,\ 0}\right\}\label{uno} \\
&=&\sup_{\substack{\lambda >0,\ x{^{\ast}\in X}^{\ast},\ \alpha
{\in\mathbb{R}}, \\ \left\Vert x^{\ast}\right\Vert\le 1/\lambda ,\
{(x^{\ast},\alpha)\in{{\rm\small epi}}\,g^{\ast}}}}{\lambda
\lbrack \left\langle x^{\ast },x\right\rangle -\alpha ]}_{+}.
\notag
\end{eqnarray}
Again with $\lambda>0$ fixed, for $x^{\ast}=0$ we observe that
\begin{eqnarray*}
\max_{(0,\alpha )\in{{\rm\small epi}}\,g^{\ast}}\lambda\left\{
\left\langle 0,x\right\rangle -\alpha \right\} &=&\max_{g^{\ast
}(0)\le\alpha }\lambda
(\left\langle 0,x\right\rangle-\alpha ) \\
&=&\lambda\big(-g^{\ast}(0)\big)\\
&\le &\lambda \inf_{x\in X}g(x) \\
&\le &\lambda g(\widehat{x})<0.
\end{eqnarray*}
According to this, the second representation in (\ref{uno})
implies the equalities
\begin{eqnarray*}
\mathrm{dist}\left(x;S\right) &=&\sup_{\substack{\lambda>0,\
\left\Vert x^{\ast}\right\Vert \leq 1/\lambda ,\  \\ {(x^{\ast
},\alpha)\in{{\rm\small epi}}\,g^{\ast}}}}{\lambda \lbrack
\left\langle x^{\ast },x\right\rangle
-\alpha]}_{+} \\
&=&\max_{{(x^{\ast },\alpha )\in{{\rm\small epi}}\,g^{\ast
}}}\sup_{\lambda
>0,\ \left\Vert x^{\ast}\right\Vert\leq 1/\lambda}{\lambda
\lbrack
\left\langle x^{\ast},x\right\rangle-\alpha ]}_{+} \\
&=&\sup_{{(x^{\ast},\alpha )\in{{\rm\small
epi}}\,g^{\ast}}}{\frac{\left[ \left\langle x^{\ast
},x\right\rangle-\alpha \right]_{+}}{\left\Vert
x^{\ast}\right\Vert },}
\end{eqnarray*}
which complete the proof of the proposition.
\end{proof}

\begin{lemma}\label{Lem_distance} Assume that SSC is satisfied for the system
$\sigma\left(p\right)$ in \eqref{eq_conv_sys}. Then for any $x\in
X$ and any $p\in l_{\infty }\left( T\right) $ we have the
representation
\begin{equation}
\mathrm{dist}\big(x;\mathcal{F}(p)\big)=\sup_{_{\left( x^{\ast
},\alpha \right)\in{{\rm\small cl}}^{\ast}C\left(p\right)}}\frac{
\left[\left\langle x^{\ast},x\right\rangle-\alpha \right]_{+}}{
\left\Vert x^{\ast}\right\Vert}. \label{eq_distance_formula}
\end{equation}
If furthermore the space $X$ is reflexive, then
\begin{equation}
\mathrm{dist}\big(x;\mathcal{F}\left(p\right)\big) =\sup_{_{\left(
x^{\ast },\alpha \right)\in C\left(p\right) }}\frac{\left[ \left\langle
x^{\ast },x\right\rangle -\alpha \right] _{+}}{\left\Vert x^{\ast
}\right\Vert }.  \label{distC(p)}
\end{equation}
\end{lemma}
\begin{proof}
We can obviously write
\begin{equation*}
\mathcal{F}\left(p\right)=\big\{x\in X\big|\;g(x)\le 0\big\}\;\mbox{ with }\;g:=\sup_{t\in T}(f_{t}-p_{t}),
\end{equation*}
where the SSC is equivalent to the existence of $\widehat{x}\in X$ such that $g(\widehat{x})<0$.

Employing further \cite[formula (2.3)]{GJL08} gives us
\begin{equation*}
\epi g^{\ast}=\epi\left\{\sup_{t\in
T}(f_{t}-p_{t})\right\}^{\ast}=\cl^*\co
\left(\bigcup\limits_{t\in T}\epi(f_{t}-p_{t})^{\ast}\right)=
\cl^*H(p),
\end{equation*}
and thus (\ref{eq_distance_formula}) comes straightforwardly from (\ref{distepig*}) together with the fact
that $\cl^{\ast}H(p)=\left[\cl^{\ast}C(p)\right]+\mathbb{R}_{+}\left(0,1\right)$ with $0\in X^{\ast}$.

Consider now the case when the space $X$ is reflexive. Arguing by contradiction, assume that (\ref{distC(p)})
does not hold and then find a scalar $\beta$ such that
\begin{equation}
\sup_{_{\left( x^{\ast},\alpha\right)\in{{\rm\small
cl}}^{\ast}C\left( p\right)}}\frac{\left[\left\langle
x^{\ast},x\right\rangle-\alpha\right] _{+}}{\left\Vert
x^{\ast}\right\Vert}>\beta>\sup_{_{\left(x^{\ast},\alpha\right)\in
C\left(p\right)}} \frac{\left[\left\langle
x^{\ast},x\right\rangle-\alpha\right]_{+}}{\left\Vert
x^{\ast}\right\Vert}. \label{doble}
\end{equation}
Thus there exists a pair $\left( \overline{x}^{\ast},\overline{\alpha }\right)
\in\cl^{\ast}C\left( p\right)$, with $\overline{{x}}{^{\ast}\in X^{\ast}\diagdown\{0\}}$ and $\overline{{\alpha }}
{\in\mathbb{R}}$, satisfying the strict inequality
\begin{equation*}
\frac{\left[\left\langle\overline{x}^{\ast },x\right\rangle-
\overline{\alpha}\right]_{+}}{\left\Vert\overline{x}^{\ast}\right\Vert}>\beta.
\end{equation*}
Since $X$ is reflexive and the set $C\left(p\right)$ is convex, the classical Mazur theorem allows us
to replace the weak$^{\ast}$ closure of $C$ by its norm closure. Hence there is
a sequence $(x_{k}^{\ast},\alpha _{k})\in C\left( p\right),\ k=1,2,...,$
converging in norm to $\left(\overline{x}^{\ast},\overline{\alpha}\right)$ with
\begin{equation*}
\lim_{k\rightarrow \infty }\frac{\left[ \left\langle x_{k}^{\ast
},x\right\rangle -\alpha _{k}\right] _{+}}{\left\Vert x_{k}^{\ast
}\right\Vert }=\frac{\left[ \left\langle\overline{x}^{\ast},x\right\rangle
-\overline{\alpha }\right] _{+}}{\left\Vert \overline{x}^{\ast}\right\Vert}
>\beta.
\end{equation*}
Therefore we find a natural number $k_{0}$ for which
\begin{equation*}
\frac{\left[ \left\langle x_{k_{0}}^{\ast },x\right\rangle -\alpha _{k_{0}}%
\right] _{+}}{\left\Vert x_{k_{0}}^{\ast }\right\Vert }>\beta.
\end{equation*}
This clearly contradicts (\ref{doble}) and thus completes the proof of the lemma.
\end{proof}

\section{Qualitative Stability via Coderivatives}

In this section we consider the parametric convex system
\eqref{eq_conv_sys} in the general framework of Banach decision
spaces $X$. The main goals of this section are to establish
necessary and sufficient conditions for the Lipschitz-like
property of the solution map \eqref{eq_fsm} to \eqref{eq_conv_sys}
and to compute the exact Lipschitzian bound of \eqref{eq_fsm} in
the general Banach space setting. As mentioned in Section~1, our
approach to these quantitative stability issues relies on reducing
the convex infinite system $\sigma (p)$ in \eqref{eq_conv_sys} to
its linearization $\widetilde{\sigma}(\rho_{p})$ in
\eqref{q_lin_sys} and then employing the corresponding results of
\cite{CLMP09} derived for linear infinite systems. This is done on
the base of coderivative analysis.

We start with deriving an upper estimate of the exact Lipschitzian
bound for the solution map \eqref{eq_fsm} by using the
aforementioned approach.

\begin{lemma}
For any $x\in X$ and any $p\in l_{\infty }\left( T\right)$ the
following holds:
\begin{equation*}
\mathrm{dist}\left(p;\mathcal{F}^{-1}\left(
x\right)\right)\ge\mathrm{dist}\left(\rho_{p};\widetilde{\mathcal{F}}^{-1}\left(
x\right)\right).
\end{equation*}
\end{lemma}

\begin{proof}
First observe that $\widetilde{\mathcal{F}}^{-1}\left(x\right)
=\varnothing $ yields
$\mathcal{F}^{-1}\left(x\right)=\varnothing$. Consider further the
nontrivial case when both sets $\widetilde{\mathcal{F}}^{-1}
\left(x\right)$ and $\widetilde{\mathcal{F}}^{-1}\left( x\right)$
are nonempty. Thus we get for any sequence $\left\{
p_{r}\right\}_{r\in \mathbb{N}}\subset l_{\infty}\left(T\right)$
that
\begin{equation*}
\mbox{dist}\left(p;\mathcal{F}^{-1}\left(x\right)\right)
=\lim_{r\in \mathbb{N}}\left\Vert p-p_{r}\right\Vert=\lim_{r\in
\mathbb{N}}\left\Vert\rho_{p}-\rho _{p_{r}}\right\Vert\ge
\mbox{dist}\left(\rho_{p};\widetilde{\mathcal{F}}^{-1}\left(
x\right)\right).
\end{equation*}
To complete the proof, recall that $p_{r}\in
\mathcal{F}^{-1}\left(x\right) $ if and only if $\rho _{p_{r}}\in
\widetilde{\mathcal{F}}^{-1}\left(x\right)$.
\end{proof}\vspace*{0.05in}

From now on we consider the nominal parameter $\overline{p}=0$,
i.e, the zero function from $T$ to $\mathbb{R}$; the corresponding
function $\rho_{ \overline{p}}$ is also the zero function from
$\widetilde{T}$ to $\mathbb{R}.\mathbb{\ }$Both zero functions
will be denoted simply by $0$.

\begin{lemma}\label{lem7}
Let $\overline{x}\in \mathcal{F}\left( 0\right)$. Then we have the
upper estimate
\begin{equation*}
\lip\mathcal{F}\left(0,\overline{x}\right)\le\lip
\widetilde{\mathcal{F}}\left(0,\overline{x}\right).
\end{equation*}
\end{lemma}

\begin{proof}
The aimed inequality comes straightforwardly from the exact
Lipschitzian bound representation (\ref{eq_quotient}) combined
with the linearized relationship (\ref{eq_equality_feasible}) and
the previous lemma.
\end{proof}\vspace*{0.05in}

The latter lemma and the results of \cite{CLMP09} for linear
infinite systems lead us to a constructive upper modulus estimate
for the original convex system.

\begin{theorem}\label{Th_1} Let $\overline{x}\in \mathcal{F}\left(
0\right)$. Assume that the SSC is satisfied for $\sigma\left(
0\right)$ and that the set $\bigcup_{t\in T}
\mathrm{dom}f_{t}^{\ast}$ is bounded in $X^*$. The following
assertions hold:

{\bf (i)} If $\overline{x}$ is a strong Slater point of $\sigma
\left( 0\right)$, then $\lip\mathcal{F}\left(0,\overline{x}\right)
=0$.

{\bf (ii)} If $\overline{x}$ is not a strong Slater point of
$\sigma \left( 0\right)$, then
\begin{equation}\label{ue}
\lip\mathcal{F}\left(0,\overline{x}\right)\le\lip
\widetilde{\mathcal{F}}\left( 0,\overline{x}\right)=\max\left\{
\|u^{\ast}\|^{-1}\big|\;\left(u^{\ast },\left\langle u^{\ast },
\overline{x}\right\rangle\right)\in\cl^{\ast}C\left(0\right)
\right\}.
\end{equation}
\end{theorem}

\begin{proof}
First of all, recall from Section~1 that the SSC property for the
convex system $\sigma(p)$ is equivalent to the SSC condition for
the linear one $\widetilde{\sigma }(\rho_{p})$. Thus the equality
in \eqref{ue} follows from \cite[Theorem~4.6]{CLMP09} under the
boundedness assumptions made in the theorem. The upper estimate in
\eqref{ue} is the content of Lemma~\ref{lem7}, and thus the proof
of the theorem is complete.
\end{proof}\vspace*{0.05in}

In what follows we show that the upper estimate in \eqref{ue}
holds in fact as {\em equality} under the assumptions of
Theorem~\ref{Th_1}. Furthermore, the boundedness assumption of
this theorem (which may be violated even in simple examples) can
be avoided in the case of reflexive decision spaces
$X$.\vspace*{0.05in}

To justify the equality in \eqref{ue}, we proceed by using {\em
coderivative analysis}. For each $t\in T$, consider a convex
function $h_{t}:l_{\infty }\left(T\right)\times X\rightarrow
\overline{\mathbb{R}}$ defined by
\begin{equation}\label{h0}
h_{t}\left(p,x\right):=\left\langle-\delta _{t},p\right\rangle
+f_{t}\left( x\right),
\end{equation}
where $\delta_{t}$ denotes the classical \emph{Dirac measure} at
$t\in T$, i.e.,
\begin{equation*}
\left\langle\delta_{t},p\right\rangle:=p_{t}\;\text{ for every
}\;p=(p_{t})_{t\in T}\in l_{\infty }\left( T\right).
\end{equation*}
It is easy to see that
\begin{equation}\label{h1}
\dom h_{t}^{\ast}=\left\{-\delta _{t}\right\}\times\dom
f_{t}^{\ast}\quad\mbox{and}\quad\gph
h_{t}^{\ast}=\left\{-\delta_{t}\right\}\times\gph f_{t}^{\ast}.
\end{equation}

The next result computes the coderivative of the solution map
\eqref{eq_fsm} to the original infinite convex system
\eqref{eq_conv_sys} in terms of its initial data. It is important
for the subsequent qualitative stability analysis conducted in
this section as well as for deriving optimality conditions in
Section~6.

\begin{proposition}\label{Prop_charact_coderiv} Let $\overline{x}\in\mathcal{F}
\left(0\right)$ for the solution map \eqref{eq_fsm} to the convex
system \eqref{eq_conv_sys}. Then $p^{\ast}\in
D^{\ast}\mathcal{F}\left(0,\overline{x}\right)\left(x^{\ast}\right)$
if and only if
\begin{equation}
\big(p^{\ast},-x^{\ast},-\left\langle x^{\ast},\overline{x}
\right\rangle\big)\in\mathrm{cl}^{\ast}\mathrm{cone}\left(\bigcup_{t\in
T} \big[\left\{-\delta_{t}\right\}\times\mathrm{gph}\,
f_{t}^{\ast}\big]\right).\label{eq_charact_coderiv}
\end{equation}
\end{proposition}

\begin{proof} Due to the obvious convexity of the graphical set
$\gph{\cal F}$ for \eqref{eq_fsm}, the cone
$N((0,\overline{x});\gph\mathcal{F})$ reduces to the classical
normal cone of convex analysis. Thus we have that $p^{\ast}\in
D^{\ast}\mathcal{F}\left(0,\overline{x}\right)\left(x^{\ast}\right)$
if and only if $\left\langle p^{\ast},p\right\rangle-\left\langle
x^{\ast},x\right\rangle\le-\left\langle
x^{\ast},\overline{x}\right\rangle$ by considering the convex
system
\begin{equation*}
\big\{h_{t}\left( p,x\right)\le 0,\;t\in T\big\}
\end{equation*}
with $h_t$ defined in \eqref{h0}. It now follows from the extended
Farkas Lemma formulated in Lemma~\ref{Lem_Farkas copy(1)} that
\begin{equation}
\big(p^{\ast},-x^{\ast},-\left\langle x^{\ast},\overline{x}
\right\rangle\big)\in\cl^*\left(\cone\dbigcup\nolimits_{t\in
T}\mathrm{epi}\,h_{t}{}^{\ast}\right).\label{Eq_charact_coderiv2}
\end{equation}

It is easy to see by applying both sides of
(\ref{Eq_charact_coderiv2}) on $\left(0,\overline{x},-1\right)$
that the epigraph in in \eqref{Eq_charact_coderiv2} can be
replaced by the graph of $h^*_t$ therein. Thus representation
(\ref{eq_charact_coderiv}) follows from that in
(\ref{Eq_charact_coderiv2}) and the expression of the graph of
$h^*_t$ given in \eqref{h1}.
\end{proof}\vspace*{0.05in}

The next important result provides a complete computation of the
coderivative norm, defined in \eqref{eq_norm coderiv}, via the
characteristic set $C(0)$ from \eqref{C}.

\begin{theorem}\label{Th_norm} Let $\overline{x}\in\mathcal{F}\left(0\right)$.
Assume that the SSC is satisfied for $\sigma\left(0\right)$ and
that the set $\bigcup_{t\in T}\mathrm{dom}\,f_{t}^{\ast}$ is
bounded in $X^*$. The following assertions hold:

{\bf (i)} If $\overline{x}$ is a strong Slater point of $\sigma
\left( 0\right)$, then $\left\Vert D^{\ast}\mathcal{F}\left(
0,\overline{x}\right)\right\Vert=0$.

{\bf (ii)} If $\overline{x}$ is not a strong Slater point of
$\sigma\left(0\right)$, then
\begin{equation*}
\left\Vert D^{\ast}\mathcal{F}\left(0,\overline{x}\right)
\right\Vert=\max\left\{\left\Vert u^{\ast}\right\Vert^{-1}\Big|\;
\big(u^{\ast},\left\langle
u^{\ast},\overline{x}\right\rangle\big)\in\cl^*C\left(0\right)\right\}>0.
\end{equation*}
\end{theorem}

\begin{proof} It follows the lines in the proof of \cite[Theorem~3.5]{CLMP09}
with using the equivalent descriptions of the strong Slater
condition for the convex inequality system \eqref{eq_conv_sys} via
the characteristic set $C(0)$ obtained in Lemma~\ref{Lemma1}.
\end{proof}\vspace*{0.05in}

Now we are ready to establish the main result of this section
containing the coderivative characterization of the Lipschitz-like
property of the solution map \eqref{eq_fsm} with the precise
computation of the exact Lipschitzian bound.

\begin{theorem}\label{Th_lip} Let $\overline{x}\in\mathcal{F}\left(0\right)$
for the solution map \eqref{eq_fsm} to the convex
inequality system $\sigma(p)$ in \eqref{eq_conv_sys} with an
arbitrary Banach decision space $X$. Then ${\cal F}$ is
Lipschitz-like around $(0,\ox)$ if and only if
\begin{equation}\label{cod-cr1}
D^*{\cal F}(0,\ox)(0)=\{0\}.
\end{equation}
If furthermore the SSC is satisfied for $\sigma\left(0\right)$ and
the set $\bigcup_{t\in T}\mathrm{dom}\,f_{t}^{\ast}$ is bounded in
$X^*$, then the following hold:

{\bf (i)} $\lip{\cal F}(0,\ox)=0$ provided that $\ox$ is a strong
Slater point of $\sigma(0)$;

{\bf (ii)} otherwise we have
\begin{equation}
\mathrm{lip}\,\mathcal{F}\left(0,\overline{x}\right)=\max \left\{
\left\Vert
u^{\ast}\right\Vert^{-1}\Big|\;\big(u^{\ast},\left\langle
u^{\ast},\overline{x}\right\rangle\big)\in \mathrm{cl}^{\ast
}C\left(0\right)\right\}>0.\label{eq_lip}
\end{equation}
\end{theorem}

\begin{proof} The ``only if" part in the coderivative criterion \eqref{cod-cr1}
is a consequence of \cite[Theorem~1.44]{mor06a} established for
general set-valued mappings of closed graph between Banach spaces.
The proof of the ``if" part in \eqref{cod-cr1} follows the lines
in the proof of \cite[Theorem~4.1]{CLMP09}.

The equality $\lip{\cal F}(0,\ox)=0$ for the exact Lipschitzian
bound in case (i) can be checked directly from the definitions
while it also follows by combining assertion (i) of
Theorem~\ref{Th_1} and assertion (i) of Theorem~\ref{Th_norm}.

It remains to justify equality \eqref{eq_lip} in the case when
$\ox$ is not a strong Slater point of $\sigma(0)$. Indeed, the
upper estimate for $\lip{\cal F}(0,\ox)$ follows from assertion
(ii) of Theorem~\ref{Th_1} and computing the coderivative norm in
assertion (ii) of Theorem~\ref{Th_norm} under the assumptions
made. The lower bound estimate
\begin{equation*}
\lip{\cal F}(0,\ox)\ge\|D^*{\cal F}(0,\ox)\|
\end{equation*}
is proved in \cite[Theorem~1.44]{mor06a} for general set-valued
mappings between Banach spaces. This completes the proof of the
theorem.
\end{proof}\vspace*{0.05in}

\begin{remark}\label{bc}
{\rm For the Lipschitzian modulus results obtained in
Theorem~\ref{Th_1} and Theorem~\ref{Th_lip} we imposed the {\em
boundedness} assumption on the set $\bigcup_{t\in
T}\mathrm{dom}\,f_{t}^{\ast}$ in the convex infinite system
\eqref{eq_conv_sys}. This corresponds to the boundedness on the
coefficient set $\{a^*_t\;|\;t\in T\}$ in the case of parametric
linear infinite systems $\{\la a^*_t,x\ra\le b_t+p_t\}$. While the
latter assumption does not look restrictive in the linear
framework, it may be too strong in the convex setting under
consideration, being violated even in some simple examples as in
the case of the following single constraint involving
one-dimensional decision and parameter variables:
\begin{equation}\label{ex}
x^2\le p\;\mbox{ for }\;x,p\in\mathbb{R}.
\end{equation}
Note that the linearized system \eqref{q_lin_sys} associated with
\eqref{ex} reads as follows:
\begin{equation*}
ux\le p+\frac{u^2}{4},\quad u\in\mathbb{R}.
\end{equation*}
In the next section we show that the aforementioned coefficient
boundedness assumption for linear systems and the corresponding
boundedness assumption on the set $\bigcup_{t\in
T}\mathrm{dom}\,f_{t}^{\ast}$ in the convex framework can be
dropped in the case of reflexive Banach spaces $X$ of decision
variables.}
\end{remark}

\begin{remark}\label{ioffe}
{\rm After the publication of \cite{CLMP09}, Alex Ioffe drew our
attention to the possible connections of some of the results
therein with those obtained in \cite{is} for general set-valued
mappings of convex graph. Examining this approach, we were able to
check, in particular, that \cite[Corollary~4.7]{CLMP09} on the
computing the exact Lipschitzian bound of linear infinite systems
via the coderivative norm under the coefficient boundedness can be
obtained by applying Theorem~3 and Proposition~5 from \cite{is}.
However, our proofs are far from being straightforward.}
\end{remark}

\section{Enhanced Stability Results in Reflexive Spaces}

In this section we primarily deal with the linear infinite system
\begin{equation}
\sigma(p):=\big\{\left\langle a_{t}^{\ast},x\right\rangle\le
b_{t}+p_{t},\ t\in T\big\}, \label{eq_linear system}
\end{equation}
where $a_{t}^{\ast}\in X^{\ast}$ and $b_{t}\in\mathbb{R}$ are
fixed for each $t$ from an arbitrary index sets $T$. Due to the
linearization approach developed above, the results obtained below
for linear systems can be translated to convex infinite systems of
type \eqref{eq_conv_sys}.

Note that in the linear case \eqref{eq_linear system} the
characteristic set \eqref{C} takes the form
\begin{equation}\label{C1}
C\left( p\right)=\mathrm{co}\big\{\left(
a_{t}^{\ast},b_{t}+p_{t}\right)\big|\;t\in T\big\} .
\end{equation}
This is our setting in \cite{CLMP09}, where the coefficient set
$\left\{a_{t}^{\ast}\;|\;t\in T\right\}\subset X^*$ is assumed to
be bounded while computing the coderivative norm $\left\Vert
D^{\ast}\mathcal{F}\left(0,\overline{x}\right)\right\Vert$ in
\cite[Theorem~3.5]{CLMP09} and the exact Lipschitzian bound $\lip
\mathcal{F}\left(0,\overline{x}\right)$ in
\cite[Theorem~4.5]{CLMP09} for the solution map ${\cal F}$ to
\eqref{eq_linear system}. In the case when $X$ is reflexive, we
are going to remove now the coefficient boundedness assumption
from both referred theorems, which implies that the boundedness of
the set $\cup_{t\in T}\mathrm{dom}f_{t}^{\ast}$ can also be
removed as an assumption throughout Section~4 when $X$ is
reflexive.

First we observe that the boundedness of the coefficients
$\left\{a_{t}^{\ast }\;|\;t\in T\right\}$ yields that only
$\varepsilon $-active indices are relevant in \eqref{C1} with
respect to the set of elements in the form $\left(u^{\ast
},\left\langle u^{\ast},\overline{x}\right\rangle\right)$
belonging to $\mathrm{cl}^{\ast}C\left(0\right)$, which from now
on is written as $\left\{ \left( \overline{x},-1\right)\right\}
^{\bot}\cap\mathrm{cl}^{\ast}C\left( 0\right)$. Given
$\overline{x}\in\mathcal{F}\left(0\right)$ and $\varepsilon\ge 0$,
we use the notation
\begin{equation*}
T_{\varepsilon }\left(\overline{x}\right):=\big\{t\in T\big|\;
\left\langle a_{t}^{\ast},\overline{x}\right\rangle\ge
b_{t}-\varepsilon\big\}
\end{equation*}
for the set of $\varepsilon $\emph{-active indices}. Let us make
the above statement precise.

\begin{proposition}\label{pr}
Assume that the coefficient set $\left\{a_{t}^{\ast}\;|\;t\in
T\right\}$ is bounded in $X^{\ast}$. Then given
$\overline{x}\in\mathcal{F}\left(0\right)$, we have the
representation
\begin{equation}
\big\{\left(\overline{x},-1\right)\big\}^{\bot}\cap\cl^{\ast}C\left(0\right)=
\bigcap_{\varepsilon>0}\cl^{\ast}\co\big\{\left(a_{t}^{\ast},
b_{t}\right)\big|\;t\in T_{\varepsilon}\left(\overline{x}\right)\big\}.
\label{eq_intersect}
\end{equation}
\end{proposition}

\begin{proof}
It follows the lines of justifying Step~1 in the proof of \cite[
Theorem~1]{CGP08}. Note that both sets in \eqref{eq_intersect} are
nonempty if and only if $\overline{x}$ is not a strong Slater
point for $\sigma\left(0\right)$. Note also that the fulfillment
of the SSC for $\sigma\left(0\right)$ in \eqref{eq_linear system}
is not required for the fulfillment of \eqref{eq_intersect}.
\end{proof}\vspace*{0.05in}

Observe that in the continuous case considered in \cite{CDLP05}
(where $T$ is assumed to be a compact Hausdorff space,
$X=\mathbb{R}^{n}$, and the mapping
$t\mapsto\left(a^*_{t},b_{t}\right)$ is continuous on $T$)
representation (\ref{eq_intersect}) reads as
\begin{equation*}
\big\{\left(\overline{x},-1\right)\big\}^{\bot}\cap C\left(0\right)
=\co\big\{\left(a^*_{t},b_{t}\right)\big|\;\ t\in T_{0}\left(\overline{x}\right)\big\}.
\end{equation*}

The following example shows that the statement of
Proposition~\ref{pr} is {\em no longer valid} without the
boundedness assumption on $\left\{a_{t}^{\ast}\;|\;t\in T\right\}$
and that in (\ref{eq_lip}) the set $\cl^{\ast}C\left(0\right)$
cannot be replaced by
$\mathrm{cl}^{\ast}\co\left\{\left(a_{t}^{\ast}, b_{t}\right)\mid
t\in T_{\varepsilon}\left(\overline{x}\right)\right\}$ for some
\emph{small} $\varepsilon>0$; i.e., it is not sufficient to
consider just $\varepsilon$-active constraints. However, the exact
bound formula (\ref{eq_lip}) {\em remains true} in this example,
with the replacement of ``$\max$" by ``$\sup$" therein.

\begin{example}
\emph{Consider the countable linear system in }$\mathbb{R}^{2}$:
\begin{equation*}
\sigma\left(p\right)=\left\{\begin{array}{ll}
\left(-1\right)^{t}tx_{1}\le 1+p_{t}, & t=1,2,..., \\
x_{1}+x_{2}\le 0+p_{0}, & t=0
\end{array}
\right\}.
\end{equation*}
\emph{The reader can easily check that}
\begin{equation*}
\mathrm{co}\big\{\left(a_{t}^{\ast},b_{t}\right)\big|\;\ t\in T_{\varepsilon
}\left( \overline{x}\right)\big\}=\big\{\left(1,1,0\right)\big\}\;\mbox{ and}
\end{equation*}
\begin{equation*}
\big\{\left(\overline{x},-1\right)\big\}^{\bot}\cap\cl
^{\ast}C\left(0\right)=\big\{\left(\alpha,1,0\right),\;\alpha\in
\mathbb{R}\big\}
\end{equation*}
\emph{ for }$\overline{x}=0_{2}$\emph{\ and
}$0\le\varepsilon<1$.\emph{ Moreover}
\begin{equation*}
\mathcal{F}\left(p\right)=\left\{0\right\}\times\left(-\infty,p_{0}
\right]\;\text{\emph{\ whenever }}\;\left\Vert p\right\Vert\le 1,
\end{equation*}
\emph{which easily implies that }
$\lip\mathcal{F}\left(0,\overline{x}\right)=1$, \emph{\ and hence
the exact bound formula (\ref{eq_lip}) holds in this example.
Observe however that for }$0<\varepsilon<1$ \emph{we have}
\begin{equation*}
\max\left\{\left\Vert
u^{\ast}\right\Vert^{-1}\big|\;\left(u^{\ast},\left\langle
u^{\ast},\overline{x}
\right\rangle\right)\in\mathrm{cl}^{\ast}\mathrm{co}\big\{\left(a_{t}^
{\ast},b_{t}\right)\big|\;t\in T_{\varepsilon
}\left(\overline{x}\right)\big\}\right\}=\frac{1}{\sqrt{2}},
\end{equation*}
{\rm which shows that $T(\ox)$ cannot be replaced by $T_\ve(\ox)$ in \eqref{eq_lip}.}
\end{example}

As we mentioned above, it is clear that
$\left\{\left(\overline{x},-1\right) \right\}^{\bot
}\cap\mathrm{cl}^{\ast}C\left(0\right)=\varnothing$ when
$\overline{x}$ is a strong Slater point for $\sigma
\left(0\right)$. The following example (where the SSC is satisfied
for $\sigma\left(0\right)$) shows that the set
$\left\{\left(\overline{x},-1\right)\right\}^{\bot}\cap
\mathrm{cl}^{\ast}C\left(0\right)$ may be empty when
$\overline{x}\in\mathcal{F}\left(0\right)$ is not a strong Slater
point for $\sigma\left(0\right)$. According to
Theorem~\ref{Th_lip}, this cannot be the case when the coefficient
set $\left\{a_{t}^{\ast}\;|\;t\in T\right\}$ is bounded. Observe
however that in this example we have
$\lip\mathcal{F}\left(0,\overline{x}\right)=0$, and thus
(\ref{eq_lip}) still holds under the convention that $\sup
\varnothing:=0$.

\begin{example}
\emph{Consider the infinite linear system in }$\mathbb{R}$:
\begin{equation*}
\sigma\left(p\right)=\Big\{tx\le\frac{1}{t}+p_{t},~t\in\left[
1,\infty\right)\Big\}
\end{equation*}
\emph{and take }$\overline{x}=0.$\emph{\ It is easy to see that}
$\left\{\left(\overline{x}
,-1\right)\right\}^{\bot}\cap\mathrm{cl}^{\ast}C\left(0\right)=\varnothing$.
{\rm Let us now check that
$\lip\mathcal{F}\left(0,\overline{x}\right)=0$. Indeed,
representation (\ref{eq_quotient}) yields
\begin{eqnarray*}
\lip\mathcal{F}\left(0,\overline{x}\right)&=&\underset{\left(
p,x\right)\rightarrow \left(0,0\right)}{\lim\sup
}\frac{\mbox{dist}\left(x;\mathcal{F}\left(p\right)\right)}{\mbox{dist}\left(p;
\mathcal{F}^{-1}\left(x\right)\right)}=\underset{\left(p,x\right)
\rightarrow\left(0,0\right)}{\lim\sup}\frac{\left[x-\inf_{t\ge
1}\left(\frac{1}{t^{2}}+\frac{p_{t}}{t}\right)\right]_{+}}{\sup_{t\ge
1}\left[ tx-\frac{1}{t}-p_{t}\right]_{+}}\\
&=&\underset{\left(p,x\right)\rightarrow \left( 0,0\right) }{\lim
\sup}\frac{\sup_{t\geq 1}\left[
x-\frac{1}{t^{2}}-\frac{p_{t}}{t}\right] _{+}}{\sup_{t\ge 1}\left[
tx-\frac{1}{t}-p_{t}\right]_{+}}.
\end{eqnarray*}
Taking into account that $\sup_{t\ge 1}\left[ tx-\frac{1}{t}-p_{t}
\right]_{+}=\infty$ if $x>0$ for every $p\in l_{\infty}\left(
\left[1,\infty\right)\right)$ and that for any $\left(p,x\right)
\in \varepsilon B_{l_{\infty }\left(\left[ 1,\infty \right)
\right)}\times\left[ -\varepsilon ,0\right]$ with $0<\varepsilon
\le 1$ we have $x-\frac{1}{t^{2}}-\frac{p_{t}}{t}\le 0$, it
follows that
\begin{equation*}
\lip\mathcal{F}\left(0,\overline{x}\right)=\underset{\left(
p,x\right)\rightarrow\left(0,0\right)}{\lim\sup}\frac{\sup_{t\ge
1/\varepsilon }\left[ x-\frac{1}{t^{2}}-\frac{p_{t}}{t}\right]
_{+}}{\sup_{t\ge 1/\varepsilon}\left[tx-\frac{1}{t}-p_{t}\right]
_{+}}\leq\varepsilon.
\end{equation*}
Since this holds for any $\varepsilon\in\left(0,1\right]$, we get
$\lip\mathcal{F}\left(0,\overline{x}\right)=0$ and thus conclude
our consideration in this example.}
\end{example}

Now we are ready to establish our major result in the case of
reflexive decision spaces $X$ in \eqref{eq_linear system}. Recall
that in this case the weak$^*$ closure $\mathrm{cl}^{\ast}\,S$ and
the norm closure $\mathrm{cl}\,S$ in $X^*$ agree for convex
subsets $S\subset X^*$.

\begin{theorem}
\label{Th_without boundedness}Assume that $X$ is reflexive and let
$ \overline{x}\in\mathcal{F}\left(0\right)$. If the SSC is
satisfied for $\sigma\left(0\right)$ in \eqref{eq_linear system},
then we have
\begin{equation}
\lip\mathcal{F}\left(0,\overline{x}\right)=\left\Vert D^{\ast}
\mathcal{F}\left(0,\overline{x}\right)\right\Vert=\sup
\left\{\left\Vert u^{\ast}\right\Vert^{-1}\Big|\;\left(
u^{\ast},\left\langle u^{\ast},\overline{x}\right\rangle
\right)\in\cl C\left(0\right)\right\}\label{eq_chain_ref}
\end{equation}
with $C(0)$ defined in \eqref{C1}, under the convention that
$\sup\varnothing:=0$.
\end{theorem}

\begin{proof} As mentioned above, the inequality $\lip\mathcal{F}
\left(0,\overline{x}\right)\ge\left\Vert D^{\ast}\mathcal{F}\left(
0,\overline{x}\right) \right\Vert$ holds for general set-valued
mappings due to \cite[Theorem 1.44]{mor06a}. Let us next consider
the nontrivial case
$\left\{\left(\overline{x},-1\right)\right\}^{\bot}\cap\cl
C\left(0\right)\ne\varnothing$ and show that
\begin{equation}
\left\Vert D^{\ast}\mathcal{F}\left(0,\overline{x}\right)
\right\Vert\geq\sup\left\{\left\Vert u^{\ast}\right\Vert
^{-1}\Big|\;\left( u^{\ast},\left\langle
u^{\ast},\overline{x}\right\rangle\right)\in\cl C\left( 0\right)
\right\}.\label{eq_101}
\end{equation}
To proceed, take $u^{\ast}\in X^{\ast}$ such that
$\left(u^{\ast},\left\langle u^{\ast},\overline{x}\right\rangle
\right)\in\cl C\left(0\right)$. The fulfillment of the SSC for
$\sigma\left(0\right)$ in \eqref{eq_linear system} ensures that
$u^{\ast}\ne 0$ according to Lemma~\ref{Lemma1}. By the latter
inclusion, find a sequence $\left\{\lambda_{k}\right\}_{k\in
\mathbb{N}}$ with $\lambda _{k}=\left(\lambda _{tk}\right)_{t\in
T} \in\mathbb{R}_{+}^{\left(T\right)}$ and $\sum_{t\in T}\lambda
_{tk}=1$ as $k\in\mathbb{N}$ satisfying
\begin{equation}
\big( u^{\ast},\left\langle
u^{\ast},\overline{x}\right\rangle\big)=\lim_{k\rightarrow \infty
}\sum_{t\in T}\lambda_{tk}\left(a_{t}^{\ast},b_{t}\right).
\label{eq_102}
\end{equation}
Since the sequence $\left\{\left\Vert
u^{\ast}\right\Vert^{-1}\sum_{t\in T}\lambda_{tk}\left(-\delta
_{t}\right)\right\}_{k\in\mathbb{N}}$ is contained in $\left\Vert
u^{\ast}\right\Vert ^{-1}B_{l_{\infty }\left( T\right)}$, the
classical Alaoglu-Bourbaki theorem ensures that a certain subnet
of this sequence (indexed by $\nu\in\mathcal{N}$) weak$^{\ast}$
converges to some $p^{\ast }\in l_{\infty}\left(T\right)^{\ast}$
with $\left\Vert p^{\ast}\right\Vert\le\left\Vert
u^{\ast}\right\Vert^{-1}$. Denoting by $e\in l_{\infty}\left(
T\right)$ the function whose coordinates are identically one, we
get
\begin{equation*}
\left\langle p^{\ast},-e\right\rangle=\lim_{\nu\in
\mathcal{N}}\left\Vert u^{\ast}\right\Vert^{-1}\sum_{t\in
T}\lambda_{t\nu }=\left\Vert u^{\ast }\right\Vert ^{-1},
\end{equation*}
and hence $\left\Vert p^{\ast}\right\Vert=\left\Vert
u^{\ast}\right\Vert^{-1}.$ Appealing now to (\ref{eq_102}) gives
us, for the subnet under consideration, the equality
\begin{equation*}
\left( p^{\ast},\left\Vert u^{\ast}\right\Vert ^{-1}u^{\ast
},\left\langle\left\Vert
u^{\ast}\right\Vert^{-1}u^{\ast},\overline{x}\right\rangle\right)
=w^{\ast}\text{-}\lim_{\nu\in\mathcal{N}}\left\Vert
u^{\ast}\right\Vert^{-1}\sum_{t\in T}\lambda_{t\nu}\left(-\delta
_{t},a_{t}^{\ast},b_{t}\right).
\end{equation*}
Employing further the coderivative description from
Proposition~\ref{Prop_charact_coderiv} yields
\begin{equation*}
p^{\ast}\in D^{\ast}\mathcal{F}\left( 0,\overline{x}\right)\left(
-\left\Vert u^{\ast}\right\Vert^{-1}u^{\ast}\right),
\end{equation*}
which implies by definition (\ref{eq_norm coderiv}) of the
coderivative norm that
\begin{equation*}
\left\Vert D^{\ast}\mathcal{F}\left( 0,\overline{x}\right)
\right\Vert\ge\left\Vert p^{\ast}\right\Vert=\left\Vert
u^{\ast}\right\Vert^{-1}.
\end{equation*}
Since $u^{\ast}$ was arbitrarily chosen from those satisfying
$\left(u^{\ast},\left\langle u^{\ast},\overline{x}\right\rangle
\right)\in\cl C\left(0\right)$, we arrive at the lower estimate
(\ref{eq_101}) for the coderivative norm.

Now let us prove the upper estimate for the exact Lipschitzian
bound
\begin{equation}\label{up-lip}
\lip\mathcal{F}\left( 0,\overline{x}\right)\le\sup \left\{
\left\Vert
u^{\ast}\right\Vert^{-1}\Big|\;\left(u^{\ast},\left\langle
u^{\ast},\overline{x}\right\rangle\right)\in\cl C\left(0\right)
\right\},
\end{equation}
which ensures, together with the lower estimates above, the
fulfillments of both equalities in \eqref{eq_chain_ref}. Arguing
by contradiction, find $\alpha>0$ such that
\begin{equation}
\lip\mathcal{F}\left(0,\overline{x}\right)>\alpha>\sup
\left\{\left\Vert
u^{\ast}\right\Vert^{-1}\Big|\;\left(u^{\ast},\left\langle
u^{\ast},\overline{x}\right\rangle\right)\in\cl C\left(0\right)
\right\}.\label{eq_103}
\end{equation}
According to the first inequality of (\ref{eq_103}), there are
sequences $p_{r}=\left( p_{tr}\right)_{t\in T}\rightarrow 0$ and
$x_{r}\rightarrow\overline{x}$ such that
\begin{equation}
\mbox{dist}\big(x_{r};\mathcal{F}(p_{r})\big)
>\alpha\,\mbox{dist}\big(p_{r};\mathcal{F}^{-1}\left(
x_{r}\right)\big)\;\text{ for all }\;r\in\mathbb{N}.
\label{eq_104}
\end{equation}
By the SSC for $\sigma\left(0\right)$ we have that
$\mathcal{F}\left( p_{r}\right)\ne\varnothing$ for $r$
sufficiently large (say for all $r$ without loss of generality).
This SSC is equivalent to the Lipschitz-like property of the
corresponding solution map $\mathcal{F}$ around $\left(
0,\overline{x}\right)$ and also to the inner/lower semicontinuity
of $\mathcal{F} $ around $\overline{x}$ by
\cite[Theorem~5.1]{DGL08}, which entails that
\begin{equation}
\lim_{r\rightarrow\infty}\mbox{dist}\big(x_{r};\mathcal{F}\left(p_{r}\big)
\right)=0.\label{semicontinf}
\end{equation}
Moreover, it follows from (\ref{eq_104}) that the quantity
\begin{eqnarray}
\mbox{dist}\big(p_{r};\mathcal{F}^{-1}\left(x_{r}\right)\big)
&=&\sup_{t\in T}\left[\left\langle a_{t}^{\ast},x_{r}\right\rangle
-b_{t}-p_{tr}\right]_{+}\label{nueva}\\
&=&\sup_{\left( x^{\ast },\alpha\right)\in C\left(p_{r}\right)
}\left[\left\langle x^{\ast},x_{r}\right\rangle-\alpha\right]
_{+}\notag
\end{eqnarray}
is finite. It follows from Lemma~\ref{Lem_distance} while
$\left\Vert p_{r}\right\Vert\le\eta$, $r=1,2,...$, that
\begin{equation*}
\mbox{dist}\big(x_{r};\mathcal{F}\left(p_{r}\right)\big)=\sup_{
_{\substack{ {x^{\ast}\in X^{\ast}\diagdown \{0\},\ \alpha\in
\mathbb{R},}\\\left(x^{\ast},\alpha \right)\in C\left(
p_{r}\right)}}}\frac{\left[\left\langle
x^{\ast},x_{r}\right\rangle-\alpha \right]_{+}}{\left\Vert
x^{\ast}\right\Vert},\quad r=1,2,... .
\end{equation*}
This allows us to find $\left(x_{r}^{\ast},\alpha_{r}\right)\in
C\left( p_{r}\right)\diagdown\{0\}$ as $r=1,2,...$ satisfying
\begin{equation}
0\le\mbox{dist}\big(x_{r},\mathcal{F}\left(p_{r}\right)\big)-
\frac{\left\langle x_{r}^{\ast },x_{r}\right\rangle-\alpha
_{r}}{\left\Vert x_{r}^{\ast}\right\Vert
}<\frac{\mbox{dist}\big(p_{r};\mathcal{F}^{-1}\left(x_{r}\right)
\big)}{r}. \label{unosobrer}
\end{equation}
Furthermore, by (\ref{eq_104}) and (\ref{nueva}) we can choose
$\left(x_{r}^{\ast},\alpha_{r}\right)$ in such a way that
\begin{equation}
\alpha\,\mbox{dist}\big(p_{r};\mathcal{F}^{-1}\left(x_{r}\right)
\big)<\frac{\left\langle
x_{r}^{\ast},x_{r}\right\rangle-\alpha_{r}}{ \left\Vert
x_{r}^{\ast}\right\Vert}+\frac{\mbox{dist}\big(p_{r};\mathcal{F}^{-1}\left(x_{r}\right)
\big)}{r}\le\frac{\mbox{dist}\big(p_{r};
\mathcal{F}^{-1}\left(x_{r}\right)\big)}{\left\Vert
x_{r}^{\ast}\right\Vert }.\label{clave}
\end{equation}
Since dist$(p_{r};\mathcal{F}^{-1}\left(x_{r}\right))>0$, we
deduce from (\ref{clave}) that
\begin{equation*}
\alpha<\frac{1}{\|x^*_r\|}+\frac{1}{r}\;\mbox{ and }0<\left\Vert
x_{r}^{\ast}\right\Vert<\frac{1}{\alpha-r^{-1}}\;\mbox{ for all
}\;r=1,2,...,
\end{equation*}
and thus, by the weak$^*$ sequential compactness of the unit ball
in dual to reflexive spaces, select a subsequence $\left\{
x_{r_{k}}^{\ast}\right\}_{k\in \mathbb{N}}$, which weak$^*$
converges to some $x^{\ast}\in X^*$ satisfying $\left\Vert
x^{\ast}\right\Vert\le 1/\alpha$. Then we get from
(\ref{semicontinf}) and (\ref{unosobrer}) that
\begin{equation*}
\lim_{k\in \mathbb{N}}\frac{\left\langle
x_{r_{k}}^{\ast},x_{r_{k}}\right\rangle-\alpha
_{r_{k}}}{\left\Vert x_{r_{k}}^{\ast}\right\Vert}=0,
\end{equation*}
which implies in turn that
\begin{equation*}
\lim_{k\in\mathbb{N}}\big(\left\langle
x_{r_{k}}^{\ast},x_{r_{k}}\right\rangle-\alpha_{r_{k}}\big)=0.
\end{equation*}
Since the sequence $\left\{x_{r_{k}}\right\}_{k\in\mathbb{N}}$
converges (in norm) to $\overline{x}$, the latter implies that
\begin{equation*}
\lim_{k\in \mathbb{N}}\alpha_{r_{k}}=\lim_{k\in\mathbb{N}
}\left\langle x_{r_{k}}^{\ast},x_{r_{k}}\right\rangle
=\left\langle x^{\ast},\overline{x}\right\rangle.
\end{equation*}
Taking into account for each $k\in\mathbb{N}$ we have
$\left(x_{r_{k}}^{\ast},\alpha _{r_{k}}\right)\in
C\left(p_{r_{k}}\right)$,  there exist $\lambda_{r_{k}}=(\lambda
_{tr_{k}})_{t\in T}$ such that $\lambda_{tr_{k}}\ge 0$, only
finitely many of them are positive,
\begin{equation*}
\sum_{t\in T}\lambda_{tr_{k}}=1,\;\mbox{ and
}\;(x_{r_{k}}^{\ast},\alpha _{r_{k}})=\sum_{t\in T}\lambda
_{tr_{k}}\left(a_{t}^{\ast},b_{t}+p_{tr_{k}}\right),\quad
k\in\mathbb N.
\end{equation*}
Combining all the above gives us the relationships
\begin{eqnarray*}
\big(x^{\ast},\left\langle x^{\ast},\overline{x}\right\rangle\big)
&=&w^{\ast}\text{-}\lim_{k\in\mathbb{N}}\big(x_{r_{k}}^{\ast},
\left\langle x_{r_{k}}^{\ast},x_{r_{k}}\right\rangle\big)\\
&=&w^{\ast}\text{-}\lim_{k\in\mathbb{N}}(x_{r_{k}}^{\ast},\alpha_{r_{k}})\\
&=&w^{\ast}\text{-}\lim_{k\in\mathbb{N}}\sum_{t\in T}\lambda_{tr_{k}}
\left(a_{t}^{\ast},b_{t}+p_{tr_{k}}\right)\\
&=&w^{\ast}\text{-}\lim_{k\in\mathbb{N}}\sum_{t\in
T}\lambda_{tr_{k}}\left( a_{t}^{\ast},b_{t}\right)\in\cl
C\left(0\right).
\end{eqnarray*}
Observe finally that $x^{\ast}\ne 0$ because, by
Lemma~\ref{Lemma1}, the linear infinite system
$\sigma\left(0\right)$ satisfies the SSC. This allows us to
conclude that
\begin{equation*}
\sup\left\{\left\Vert
u^{\ast}\right\Vert^{-1}\big|\;\big(u^{\ast},\left\langle
u^{\ast},\overline{x}\right\rangle\big)\in\cl C\left(0\right)
\right\} \geq \left\Vert x^{\ast}\right\Vert^{-1}\ge\alpha,
\end{equation*}
which contradicts (\ref{eq_103}) and thus completes the proof of
the theorem.
\end{proof}\vspace*{0.05in}

As mentioned above, the results obtained in this section for
linear infinite systems make it possible to drop the major
boundedness assumptions imposed in the corresponding results of
Section~4 for convex infinite systems in reflexive spaces. Let us
present the improved ``reflexive" version of Theorem~\ref{Th_lip},
the main result of the preceding section.

\begin{theorem}\label{lip-convex} Let $\overline{x}\in\mathcal{F}\left(0\right)$
for the solution map \eqref{eq_fsm} to the convex inequality
system $\sigma(p)$ in \eqref{eq_conv_sys} with a reflexive Banach
decision space $X$. If the SSC is satisfied for
$\sigma\left(0\right)$, then the following hold:

{\bf (i)} $\lip{\cal F}(0,\ox)=0$ provided that $\ox$ is a strong
Slater point of $\sigma(0)$;

{\bf (ii)} otherwise we have
\begin{equation*}
\mathrm{lip}\,\mathcal{F}\left(0,\overline{x}\right)=\sup\left\{
\left\Vert
u^{\ast}\right\Vert^{-1}\Big|\;\big(u^{\ast},\left\langle
u^{\ast},\overline{x}\right\rangle\big)\in \mathrm{cl}^{\ast
}C\left(0\right)\right\},
\end{equation*}
where the characteristic set $C(0)$ is defined in \eqref{C}.
\end{theorem}

\begin{proof} Follows the lines in the proof of
Theorem~\ref{Th_lip} with the usage of Theorem~\ref{Th_without
boundedness} instead of \cite[Theorem~4.6]{CLMP09} therein.
\end{proof}

\section{Optimality Conditions for Infinite Programs}

In this section we consider an infinite (or semi-infinite)
optimization problem of type $\mathrm{(P)}$ written in the form:
\begin{equation}\label{1.1}
\inf\ph(p,x)\;\mbox{ s.t. }\;x\in{\cal F}(p),
\end{equation}
where $x\in X$, $p=(p_t)_{t\in T}\in l_\infty(T)$ with an
arbitrary index set $T$, and where the set of feasible solutions
\begin{equation}\label{1.2}
{\cal F}(p):=\big\{x\in X\big|\;f_t(x)\le p_t,\;t\in T\big\}
\end{equation}
is defined by the parameterized infinite system of convex
inequalities \eqref{eq_conv_sys} over a general Banach space $X$
satisfying the standing assumptions formulated in Section~1. We
refer the reader to \cite{CLMP10} for the justification and
valuable examples of such a two-variable version of the
infinite/semi-infinite program $\mathrm{(P)}$ in the particular
case of linear infinite inequality systems \eqref{eq_linear
system}.

The main goal of the section is to derive {\em necessary
optimality conditions} for optimal solutions to \eqref{1.1} under
general requirements on {\em nonconvex} and {\em nonsmooth} cost
functions $\ph\colon l_\infty(T)\times X\to\overline{\mathbb{R}}$.
Involving the coderivative analysis of Section~4, we obtain
optimality conditions in the general {\em asymptotic form}
developed in \cite{CLMP10} for linear infinite systems; see also
the discussions and references therein on the comparison with
other kinds of optimality conditions in semi-definite and infinite
optimization. Furthermore, we establish results of two independent
types: {\em lower subdifferential} and {\em upper subdifferential}
depending on the type of subgradients used for cost functions; see
below.\vspace*{0.05in}

Let us start with {\em lower} subdifferential conditions, which
are of the conventional type in nonsmooth {\em minimization}.
Since our infinite/semi-infinite setup is given intrinsically in
general Banach spaces by the structure of \eqref{1.2} with $p\in
l_\infty(T)$ independently of the dimension of $X$, we cannot
employ the well-developed Asplund space theory from
\cite{mor06a,mor06b}. The most appropriate subdifferential
construction in our framework is the so-called approximate
subdifferential by Ioffe \cite{iof89,iof}, which is a general
(while more complicated, topological) Banach space extension of
the (sequential) basic/limiting subdifferential by Mordukhovich
\cite{mor76,mor06a} that may be larger than the latter even for
locally Lipschitzian functions on nonseparable Asplund spaces
while it is always smaller than the Clarke subdifferential; see
\cite[Subsection~3.2.3]{mor06a} for more details.

The approximate subdifferential constructions in Banach spaces are
defined by the following multistep procedure. Given a function
$\ph\colon Z\to\oR$ finite at $\oz$, we first consider its {\em
lower Dini} (or Dini-Hadamard) {\em directional derivative}
\begin{eqnarray*}
d^{-}\ph(\oz;v):=\liminf_{u\to v,\,t\dn
0}\frac{\ph(\oz+tu)-\ph(\oz)}{t},\quad v\in Z,
\end{eqnarray*}
and then define the {\em Dini $\ve$-subdifferential} of $\ph$ at
$\oz$ by
\begin{eqnarray*}
\partial^{-}_\ve\ph(\oz):=\big\{z^*\in Z^*\big|\;\la z^*,v\ra\le d^{-}\ph(\oz;v)+\ve\|v\|\;
\mbox{ for all }\;v\in Z\big\},\quad\ve\ge 0,
\end{eqnarray*}
putting $\partial^{-}_\ve\ph(\oz):=\emp$ if $\ph(\oz)=\infty$. The
{\em $A$-subdifferential} of $\ph$ at $\oz$ is defined via
topological limits involving finite-dimensional reductions of
$\ve$-subgradients by
\begin{eqnarray*}
\partial_A\ph(\oz):=\bigcap_{L\in{\cal
L},\,\ve>0}\Bar{\Limsup_{z\st{\ph}{\to}\oz}}\;\partial^{-}_\ve(\ph+\dd_L)(z),
\end{eqnarray*}
where ${\cal L}$ signifies the collection of all the
finite-dimensional subspaces of $Z$, where $z\st{\ph}{\to}\oz$
means that $z\to\oz$ with $\ph(z)\to\ph(\oz)$, and where
$\overline{\Limsup}$ stands for the {\em topological}
Painlev\'e-Kuratowski upper/outer limit of a mapping $F\colon
Z\tto Z^*$ as $z\to\oz$ defined by
\begin{eqnarray*}
\begin{array}{ll}
\overline{\disp\Limsup_{z\to\oz}}\,F(z):=\Big\{z^*\in
Z^*\Big|&\exists\mbox{ net }\;
(z_\nu,z^*_\nu)_{\nu\in{\cal N}}\subset Z\times Z^*\;\mbox{ s.t. }\;z^*_\nu\in F(z_\nu),\\
&\disp(z_\nu,z^*_\nu)\st{\|\cdot\|\times w^*}{\to}(\oz,z^*)\Big\}.
\end{array}
\end{eqnarray*}
Then the {\em approximate $G$-subdifferential} of $\ph$ at $\oz$
(the main construction here called the ``nucleus of the
$G$-subdifferential" in \cite{iof89}) is defined by
\begin{eqnarray}\label{5.1}
\partial_G\ph(\oz):=\Big\{z^*\in X^*\Big|\;(z^*,-1)\in\disp
\bigcup_{\lm>0}\lm\partial_A{\rm{dist}}
\big((\oz,\ph(\oz));\epi\ph\big)\Big\},
\end{eqnarray}
where $\epi\ph:=\{(z,\mu)\in Z\times\R|\;\mu\ge\ph(z)\}$. This
construction, in any Banach space $Z$, reduces to the classical
derivative in the case of smooth functions and to the classical
subdifferential of convex analysis if $\ph$ is convex. In what
follows we also need the {\em singular $G$-subdifferential} of
$\ph$ at $\ox$ defined by
\begin{eqnarray}\label{5.2}
\partial_G^\infty\ph(\oz):=\Big\{z^*\in X^*\Big|\;(z^*,0)\in\disp\bigcup_{\lm>0}\lm
\partial_A{\rm{dist}}\big((\oz,\ph(\oz));\epi\ph\big)\Big\}.
\end{eqnarray}
Note that $\partial^\infty_G\ph(\oz)=\{0\}$ if $\ph$ is locally
Lipschitzian around $\ox$.\vspace*{0.05in}

Now we are ready to derive the lower subdifferential necessary
optimality conditions for problem (\ref{1.1}) with the convex
infinite constraints (\ref{1.2}) and a general nonsmooth cost
function $\ph$. These conditions and the subsequent results of
this section address an arbitrary local minimizer
$(\op,\ox)\in\gph{\cal F}$ to the problem under consideration.
Following our convention in the previous sections, we suppose
without loss of generality that $\op=0$.

\begin{theorem}\label{lower} Let $(0,\ox)\in\gph{\cal F}$ be a
local minimizer for problem {\rm(\ref{1.1})} with the constraint
system {\rm(\ref{1.2})} given by the infinite convex inequalities
$\sigma(p)$ in a Banach space $X$. Assume that the cost function
$\ph\colon l_\infty(T)\times X\to\oR$ is lower semicontinuous
around $(0,\ox)$ with $\ph(0,\ox)<\infty$. Suppose furthermore
that:

{\bf (a)} either $\ph$ is locally Lipschitzian around $(0,\ox)$;

{\bf (b)} or ${\rm{int}}(\gph{\cal F})\ne\emp$ $($which holds, in
particular, when the SSC holds for $\sigma(0)$ and the set
$\cup_{t\in T}\dom f^*_t$ is bounded$)$ and the system
\begin{eqnarray}\label{5.3}
(p^*,x^*)\in\partial^\infty_G\ph(0,\ox),\quad-\big(p^*,x^*,\la
x^*,\ox\ra\big)\in{\rm{cl}}^*\mathrm{cone}\left(\bigcup_{t\in T}
\big[\left\{-\delta_{t}\right\}\times\mathrm{gph}\,
f_{t}^{\ast}\big]\right)
\end{eqnarray}
admits only the trivial solution $(p^*,x^*)=(0,0)$.\\[1ex]
Then there is a $G$-subgradient pair
$(p^*,x^*)\in\partial_G\ph(0,\ox)$ such that
\begin{eqnarray}\label{5.4}
-\big(p^*,x^*,\la
x^*,\ox\ra\big)\in{\rm{cl}}^*\mathrm{cone}\left(\bigcup_{t\in T}
\big[\left\{-\delta_{t}\right\}\times\mathrm{gph}\,
f_{t}^{\ast}\big]\right).
\end{eqnarray}
\end{theorem}

\begin{proof} The original problem (\ref{1.1}) can be rewritten as
a mathematical program with geometric constraints:
\begin{eqnarray}\label{5.5}
\mbox{minimize }\;\ph(p,x)\;\mbox{ subject to }\;(p,x)\in\gph{\cal
F},
\end{eqnarray}
which can be equivalently described in the form of by
unconstrained minimization with ``infinite penalties":
\begin{eqnarray*}
\mbox{minimize }\;\ph(p,x)+\dd\big((p,x);\gph{\cal F}\big).
\end{eqnarray*}
By the $G$-generalized Fermat stationary rule for the latter
problem, we have
\begin{eqnarray}\label{5.6}
(0,0)\in\partial_G\big[\ph+\dd(\cdot;\gph{\cal F})\big](0,\ox).
\end{eqnarray}
Employing the $G$-subdifferential sum rule to (\ref{5.6}),
formulated in \cite[Theorem~7.4]{iof89} for the ``nuclei", gives
us
\begin{eqnarray}\label{5.7}
(0,0)\in\partial_G\ph(0,\ox)+N\big((0,\ox);\gph{\cal F}\big)
\end{eqnarray}
provided that either $\ph$ is locally Lipschitzian around
$(0,\ox)$, or the interior of $\gph{\cal F}$ is nonempty and the
qualification condition
\begin{eqnarray}\label{5.8}
\partial^\infty_G\ph(0,\ox)\cap\big[-N\big((0,\ox);\gph{\cal F}
\big)\big]=\{0,0\}
\end{eqnarray}
is satisfied. It is not hard to check (cf.\
\cite[Remark~2.4]{CLMP09}) that the strong Slater condition for
$\sigma(0)$ and the boundedness of the set $\cup_{t\in T}\dom
f^*_t$ imply that the interior of $\gph{\cal F}$ is not empty.
Observe further that, due to the coderivative definition
(\ref{cod}), the optimality condition (\ref{5.7}) can be
equivalently written as
\begin{eqnarray}\label{5.9}
\mbox{there is }\;(p^*,x^*)\in\partial_G\ph(0,\ox)\;\mbox{ with
}\;-p^*\in D^*{\cal F}(0,\ox)(x^*).
\end{eqnarray}
Employing now in (\ref{5.9}) the coderivative calculation from
Proposition~\ref{Prop_charact_coderiv}, we arrive at (\ref{5.4}).
Similar arguments show that the qualification condition
(\ref{5.8}) can be expressed in the explicit form (\ref{5.3}), and
thus the proof is complete.
\end{proof}\vspace*{0.05in}

The result of Theorem~\ref{lower} can be represented in a much
simpler form for smooth cost functions in (\ref{1.1}); it also
seems to be new for infinite programming under consideration.
Recall that a function $\ph\colon Z\to\R$ is {\em strictly
differentiable} at $\oz$, with its gradient at this point denoted
by $\nabla\ph(\oz)\in Z^*$, if
\begin{eqnarray*}
\disp\lim_{z,u\to\oz}\frac{\ph(z)-\ph(u)-\la\nabla\ph(\oz),z-u\ra}{\|z-u\|}=0,
\end{eqnarray*}
which surely holds if $\ph$ is continuously differentiable around
$\oz$. Since we have
\begin{equation*}
\partial_G\ph(0,\ox)=\big\{\big(\nabla_p\ph(0,\ox),\nabla_x\ph(0,\ox)\big)\big\}
\end{equation*}
provided that $\ph$ in \eqref{1.1} is strictly differentiable at
$(0,\ox)$ (and hence locally Lipschitzian around this point), then
condition \eqref{5.4} reduces in this case to
\begin{eqnarray}\label{sm}
-\Big(\nabla_p\ph(0,\ox),\nabla_x\ph(0,\ox),\big\la\nabla_x\ph(0,\ox),\ox\big\ra\Big)
\in{\rm{cl}}^*{\rm cone}\left(\bigcup_{t\in T}
\big[\left\{-\delta_{t}\right\}\times\mathrm{gph}\,
f_{t}^{\ast}\big]\right).
\end{eqnarray}

Next we derive qualified asymptotic necessary optimality condition
of a new {\em upper subdifferential} type, initiated in
\cite{mor04} for other classes of optimization problems with
finitely many constraints; see also \cite[Section~4]{CLMP10} for
infinite programs with linear constraints. The upper
subdifferential optimality conditions presented below are
generally independent of Theorem~\ref{lower} for problems with
nonsmooth objectives; see the discussion below. The main
characteristic feature of upper subdifferential conditions is that
they apply to {\em minimization} problems but not to the expected
framework of maximization.

To proceed, we recall the notion of the {\em Fr\'echet upper
subdifferential} (known also as the Fr\'echet or viscosity
superdifferential) of $\ph\colon X\to\R$ at $\oz$ defined by
\begin{eqnarray}\label{5.11}
\Hat\partial^+\ph(\oz):=\Big\{z^*\in
Z^*\Big|\;\disp\limsup_{z\to\oz}\frac{\ph(z)-\ph(\oz)-\la
z^*,z-\oz\ra}{\|z-\oz\|}\le 0\Big\},
\end{eqnarray}
which reduces to the classical gradient $\nabla\ph(\oz)$ if $\ph$
is Fr\'echet differentiable at $\oz$  (may not be strictly) and to
the (upper) subdifferential of {\em concave} functions in the
framework of convex analysis.

\begin{theorem}\label{upper} Let $(0,\ox)\in\gph{\cal F}$ be a local
minimizer for problem {\rm(\ref{1.1})} with the convex infinite
constraint system {\rm(\ref{1.2})} in Banach spaces. Then every
upper subgradient $(p^*,x^*)\in\Hat\partial^+\ph(0,\ox)$ satisfies
inclusion {\rm(\ref{5.4})} in Theorem~{\rm\ref{lower}}.
\end{theorem}

\begin{proof} It follows the proof of
\cite[Theorem~4.1]{CLMP10} based on the variational description of
Fr\'echet subgradients in \cite[Theorem~1.88(i)]{mor06a} and
computing the coderivative of the feasible solution map
\eqref{1.2} given in Proposition~\ref{Prop_charact_coderiv}.
\end{proof}\vspace*{0.05in}

As a consequence of Theorem~\ref{upper}, we get the simplified
necessary optimality condition \eqref{sm} for the infinite program
whose objective $\ph$ is merely {\em Fr\'echet differentiable} at
the optimal point $(0,\ox)$.

Note also that, in contrast to Theorem~\ref{lower}, we impose no
additional assumptions on $\ph$ and ${\cal F}$ in
Theorem~\ref{upper}. Furthermore, the resulting inclusion
\eqref{5.4} is proved to hold for {\em every} Fr\'echet upper
subgradient $(p^*,x^*)\in\Hat\partial^+\ph(0,\ox)$ in
Theorem~\ref{upper} instead of {\em some} $G$-subgradient
$(p^*,x^*)\in\partial_G\ph(0,\ox)$ in Theorem~\ref{lower}. On the
other hand, it occurs that $\Hat\partial^+\ph(0,\ox)=\emp$ in many
important situations (e.g., for convex objectives) while
$\partial_G\ph(0,\ox)\ne\emp$ for every local Lipschitzian
function on a Banach space. We refer the reader to
\cite[Remark~4.5]{CLMP10} and \cite[Commentary~5.5.4]{mor06b} for
extended comments on various classes of functions admitting upper
Fr\'echet subgradients and additional regularity properties
ensuring strong advantages of upper subdifferential optimality
conditions in comparison with their lower subdifferential
counterparts.\vspace*{0.05in}

{\bf Acknowledgements.} The authors are indebted to Radu Bo\c{t}
who suggested Proposition~\ref{Prop_distance} and the current
version of Lemma~\ref{Lem_distance} allowing us to remove a
certain boundedness assumption therein, which was present in the
previous version of the paper. We also very grateful to Alex Ioffe
who brought our attention to the recent paper \cite{is} and its
connections with some results of \cite{CLMP09}; see more
discussions in Remark~\ref{ioffe}.

\end{document}